\documentclass[11pt,english,reqno]{amsart}

\usepackage{latexsym}
\usepackage{mathrsfs}
\usepackage{epsfig, graphicx}
\usepackage{graphics,epsfig}

\usepackage{amsmath, appendix, ulem,amsfonts}
\usepackage[alphabetic,nobysame]{amsrefs}
\usepackage{amssymb, color}
\usepackage[margin=1in]{geometry}

\usepackage{subfigure}

\usepackage{graphicx}
\usepackage{bm}
\usepackage{subeqnarray}
\usepackage[norelsize]{algorithm2e}
\usepackage{algpseudocode}
\usepackage{cases}

\theoremstyle{plain}
\newtheorem{theorem}{Theorem}

\theoremstyle{definition}

\newtheorem{remark}[theorem]{Remark}

\textheight 8.5in \textwidth 6.5in \numberwithin{equation}{section}

\title[Implicit AP method for Linear Transport Equation]{Implicit asymptotic preserving method for linear transport equations}

\author{Qin Li} 
\address{Mathematics Department, University of Wisconsin-Madison, 480 Lincoln Dr., Madison, WI 53705}
\email{qinli@math.wisc.edu}
\author{Li Wang}
\address{Departments of Mathematics, State University of New York at Buffalo, 244 Mathematics Building, Buffalo, NY 14260}
\email{lwang46@buffalo.edu}

\date{\today}

\thanks{We would like to thank Professors Shi Jin, Jim Morel and Cory Hauck for fruitful discussions. And we would also like to acknowledge the generous support from KI-net during the second author's visit at UW-Madison.}

\newcommand{\eps}{\varepsilon}
\newcommand{\half}{\frac{1}{2}}

\newcommand{\RR}{{\mathbb{R}}}
\newcommand{\Amat}{{\mathsf{A}}}
\newcommand{\Bmat}{\mathsf{B}}
\newcommand{\Cmat}{\mathsf{C}}
\newcommand{\Pmat}{\mathsf{P}}
\newcommand{\Imat}{\mathsf{I}}
\newcommand{\Aop}{{\mathcal{A}}}
\newcommand{\Bop}{\mathcal{B}}
\newcommand{\Cop}{\mathcal{C}}
\newcommand{\Pop}{\mathcal{P}}

\newcommand{\Lop}{\mathcal{L}}
\newcommand{\Dop}{\mathcal{D}}

\newcommand{\dt}{\partial_t}
\newcommand{\dx}{\partial_x}

\newcommand{\average}[1]{\left\langle#1\right\rangle}

\newcommand{\Average}[1]{ \vert\!\vert\!\vert#1 \vert\!\vert\!\vert}
\newcommand{\tnorm}[1]{\|#1\|}

\newcommand{\Dx}{\Delta x}
\newcommand{\Dt}{\Delta t}

\begin{document}
\maketitle
\begin{abstract}
The computation of the radiative transfer equation is expensive mainly due to two stiff terms: the transport term and the collision operator. The stiffness in the former comes from the fact that particles (such as photons) travel at the speed of light, while that in the latter is due to the strong scattering in the diffusive regime. We study the fully implicit scheme for this equation to account for the stiffness. The main challenge in the implicit treatment is the coupling between the spacial and velocity coordinates that requires the large size of the to-be-inverted matrix, which is also ill-conditioned and not necessarily symmetric. Our main idea is to utilize the spectral structure of the ill-conditioned matrix to construct a pre-conditioner, which, along with an exquisite split of the spatial and angular dependence, significantly improve the condition number and allows matrix-free treatment. We also design a fast solver to compute this pre-conditioner explicitly in advance. Meanwhile, we reformulate the system via an even-odd parity, which results in a symmetric and positive definite matrix that can be inverted using conjugate gradient method. This idea can also be implemented to the original non-symmetric system whose inversion is solved by GMRES. A qualitative comparison with the conventional methods, including Krylov iterative method pre-conditioned with diffusive synthetic acceleration and asymptotic preserving scheme via even-odd decomposition, is also discussed. 
\end{abstract}

\section{Introduction}
The linear transport equation describes the physical process of interaction of radiation with background material such as radiative transfer, neutron transport, and etc. It often contains a diffusive scaling that accounts for the strong scattering effect and leads to diffusion equations. On the contrary, when the scattering is weak, the propagation of radiation is almost a free transport with the speed of light, which is named as free streaming limit. In practice, the material usually contains both strong and weak scattering regimes, and thereby it is desirable to design a numerical method that is uniformly accurate in both cases without resolving the small scales. 

Asymptotic preserving (AP) schemes arise to serve this purpose. As the name states, it preserves the asymptotic limit at the discrete level. More specifically, such method, when applied to certain equations with small parameters, should automatically becomes a stable solver for the corresponding limit equations without resolving the mesh size and time step. In the context of steady neutron transport, AP scheme was first studied by Larsen and Morel~\cite{LM89} and then Jin and Levermore~\cite{JL93}, and Klar \cite{Klar98}. A rigorous convergence analysis was subsequently carried out by Golse, Jin and Levermore~\cite{GJL99}. For time-dependent transport problem, a decomposition of the distribution function is often performed, either via a macro-micro decomposition~\cite{LM08} or an even-odd decomposition~\cite{JPT00, KFJ15}. Upon such decomposition, the stiff and non-stiff term get separate off and an implicit-explicit scheme is applied to have the two terms treated respectively. The idea was later extended to a higher order implementation by Boscarino et. al.~\cite{BPR13}. Another related approach is termed the unified gas kinetic scheme framework, which was first proposed by Mieussens for linear transport equation~\cite{Mieussens13} and recently modified by Sun et. al to treat the nonlinear problem~\cite{SJX15}. 

Nevertheless, most AP schemes mentioned above, still suffer from a restrictive parabolic CFL condition that comes from the diffusion limit. Additionally, they are designed for diffusive scaling and thus may fail to capture the free streaming limit at which the radiation travels at the speed of light. Therefore, a fully implicit method is desired to remedy both problems. However, the main challenge that prevents researchers from directly applying the fully implicit method is the inversion of a large size of matrix due to the high dimension and the coupling between the spatial and velocity coordinates. Even worse, this matrix is often ill-conditioned in some regimes and makes the inversion impractical. 

Constant efforts have been made in the past few decades towards developing efficient implicit solver, both deterministic and stochastic. For deterministic method, the current state-of-the-art is the Krylov iterative method for the discrete-ordinate system preconditioned by diffusion synthetic acceleration (DSA) \cite{WWM04}. Besides the ray-effects \cite{MWLP03} that generates by the discrete ordinate method, this preconditioned Krylov iteration has two drawbacks. One comes from the sweep which is used to invert the convection term. As it is implemented along the direction of meshes, it is hard to implement in parallel and on unstructured spatial grid in high dimensions. The other is due to the complication of the method. Since each time iteration includes a sub-iterations of sweeps and diffusion solvers for DSA pre-conditioner,  it is very complicated to extend to nonlinear case.  For stochastic method, the most widely used one is the implicit Monte Carlo (IMC) method \cite{FC71}. But it is expensive because accurate solutions require an adequate sampling of points in phase space, and it also suffers from unavoidable fluctuations. Parallel to these two approaches is the development of moment method, which approximates the solution in terms of spherical harmonics, and results in a moment system named $P_N$ or $SP_N$ system. The foundation of this approximation was laid in \cite{Pomraning93, Larsen80}, and later an explosion of new developments emerges in two sub areas: one emphasis on the derivation of simpler but accurate version of moment system such as \cite{TL96, BL01, FS11, OLS13} and the other focuses primarily on designing applicable numerical methods such as \cite{FKLY07, Hauck11}, to name just a few.

In this paper, we develop a highly efficient implicit solver for the linear transport equation that is uniformly accurate across different regimes: diffusive and free streaming. The main idea is to build a pre-conditioner based on the spectral structure of the collision, which is the main source of ill-conditioning. Then we reformulate the original system using an even-odd decomposition such that the resulting matrix to be inverted, upon discretization, is symmetric and positive definite, and thereby the conjugate gradient method can be used. This method is amenable to parallelization, and can be easily extended to anisotropic scattering. 

In the following, we recall two state-of-the-art schemes that motivates our idea and will be compared to some extent with our newly designed scheme. To begin with, we summarize the typical equations we consider. In slab geometry, the radiation intensity $f(t,x, \mu)$ solves the following dimensionless equation: 
\begin{equation} \label{system: originalTransport}
\dt f + \frac{\mu}{\eps} \dx f = \frac{\sigma}{\eps^2}(\rho - f)\,,
\end{equation}
where $\sigma (x) $ is the scattering cross-section and $\eps$ is the Knudsen number. $\mu$ is the cosine of the angle between a particle's direction of flight and the $x$-axis. Here the external source and absorption effect are neglected for simplicity. $\rho = \average{f} = \frac{1}{2}\int_{-1}^1 f d\mu$ denotes the density. In the zero limit of $\eps$, one has:
\begin{equation} \label{eqn:average}
\dt \rho - \dx \left( \frac{1}{3\sigma} \dx \rho\right) = 0\,,
\end{equation}
which is termed the diffusion limit. In planar geometry, the equation rewrites as
\begin{equation} \label{transport_2D}
\dt f + \frac{\Omega}{\eps} \cdot \nabla_x f = \frac{\sigma}{\eps^2}(\rho - f)\,,
\end{equation}
where $\Omega$ is the angular variable representing the two direction-of-flight; $\Omega = (\xi, \eta)$, $-1 \leq \xi, \eta \leq 1$ and $\eta^2 + \xi^2 = 1$. Here $\rho = \frac{1}{2\pi} \int_{|\Omega| = 1} f d\Omega $, and the corresponding diffusion limit is 
\begin{eqnarray} \label{diff_2D}
\rho_t - \frac{1}{2} \nabla \cdot \left( \frac{1}{\sigma} \nabla \rho \right) = 0.
\end{eqnarray}

\subsection{Preconditioned Krylov method}
The preconditioned Krylov method is born from $S_N$ source iteration (SI) with diffusion synthetic acceleration (DSA), but outperformed SI in computing diffusive problems with discontinuous material properties. We first recall SI with DSA to put it in the context \cite{AL02}. Consider a semi-discrete version of (\ref{system: originalTransport})
\begin{equation}\label{si-1}
f^{n+1} + \frac{\mu \Dt}{\varepsilon} \partial_xf^{n+1} - \frac{\sigma \Dt}{\varepsilon^2}(\rho^{n+1} - f^{n+1}) = f^n\,.
\end{equation}
where both the convection and scattering terms are treated implicitly, which results in an inversion of a large linear system. Due to the sparsity of the system, SI solves it iteratively that resembles the Richardson iteration. Denote 
\begin{eqnarray}
\mathcal{L} = 1 + \frac{\mu \Dt }{\eps} \partial_x+ \frac{\sigma \Dt}{\eps^2},  \qquad  \mathcal{P} = \average{\cdot} 
\end{eqnarray}
as two operators. Given $f^n$ and $\rho^n$ the solution at time $t^n$, let $f^{(l)}$ be the approximation of $f^{n+1}$ after $l$ iterations, then it solves
\begin{equation}\label{si-2}
\Lop f^{(l+1)}= \frac{\sigma \Delta t}{\varepsilon^2} \Pop f^{(l)}  + f^n \,.
\end{equation}
which is the main iteration in SI.  When $\eps$ is small, the convergence of $f^{(l)}$ is rather slow; this is because the decreasing rate of the error $\frac{\delta f^{(l+1)}}{\delta f^{(l)}}$  (where $\delta f^{(l)} = f - f^{(l)}$ and $f$ denotes the exact solution to (\ref{si-1})) is close to one. Then DSA helps to accelerate the convergence. Specifically, consider the equation for $\delta f^{(l)}$:
$$ \Lop (\delta f^{(l+1)}) = \frac{\sigma \Dt}{\eps^2} \Pop(\delta f^{(l+1)}) + \frac{\sigma \Dt}{\eps^2}\left( \rho^{(l+1)} - \rho^{(l)}\right),$$
here $\rho^{(l)} = \Pop(f^{(l)})$ and  $\delta \rho^{(l)} = \Pop(\delta f^{(l)})$. This equation is as hard to solve as (\ref{si-1}) and the key idea is to use a diffusive approximation
$$\delta \rho^{(l+\half)} - \frac{\Dt}{3} \dx \left( \frac{1}{\sigma} \dx \delta \rho^{(l+\half)} \right) =  \frac{\sigma \Dt}{\eps^2}  \left( \rho^{(l+\half)} - \rho^{(l)} \right),$$
which is much easier to compute. Then SI with DSA can be summarized as follows. 
\begin{equation} \label{SI-DSA}
\left\{
\begin{aligned}
&f^{(l + \half)}= \frac{\sigma \Delta t}{\varepsilon^2}  \Lop^{-1} \rho ^{(l)} + \Pop \Lop^{-1} f^n \, ,
\\
&\delta \rho^{(l+\half)}= \frac{\sigma \Delta t}{\eps^2}(1+\Dop)^{-1} \left( \rho^{(l+\half)} - \rho^{(l)} \right),
\\
&\rho^{(l+1)} = \rho^{(l+\half)} + \delta \rho^{(l+\half)},
\end{aligned}
\right.
\end{equation}
where $\delta \rho^{(l+\half)} = \rho - \rho^{(l+\half)}$ and $\rho$ is again the exact solution to (\ref{si-1}) and $\Dop = -\frac{\partial}{\partial_x}\left( \frac{\Dt}{3\sigma} \frac{\partial}{\partial_x}\right)$ is the diffusion operator.

As DSA is shown to be inefficient in heterogeneous multi-dimensional calculations \cite{Azmy02}, a Krylov method is developed with DSA as a preconditioner. The idea is to use more advanced iterative technique than Richardson iteration to solve
\begin{eqnarray} \label{pk-1}
\left( \mathcal{L} - \frac{\sigma \Dt}{\eps^2} \mathcal{P} \right) f^{n+1} = f^n.
\end{eqnarray}
As mentioned above, the condition number of the corresponding discretization matrix of $ \mathcal{L} - \frac{\sigma \Dt}{\eps^2} \mathcal{P}$  can be very big when $\eps$ is small, and thus a preconditioner is needed. Although an optimal preconditioner is not necessarily obvious, a good choice \cite{WWM04} would be to first multiply (\ref{pk-1}) by the operator $\mathcal{P}\mathcal{L}^{-1}$, which leads to
\begin{eqnarray}\label{pk-2}
\left( 1 - \frac{\sigma \Dt}{\eps^2} \mathcal{P} \mathcal{L}^{-1}\right) \rho^{n+1} = \mathcal{P}\mathcal{L}^{-1} f^n.
\end{eqnarray}
Then an operator including a diffusion is expected to further precondition the above system. Note from a manipulation of the equations in (\ref{SI-DSA}) and a comparison with (\ref{si-2}), a reasonable pre-conditioner is $1 +  \frac{\sigma \Dt}{\eps^2} (1+\Dop)^{-1}$. Multiply it with~\eqref{pk-2}, one has
\begin{eqnarray}\label{pk-3}
\left(1 +  \frac{\sigma \Dt}{\eps^2} (1+\Dop)^{-1} \right) \left( 1 - \frac{\sigma \Dt}{\eps^2} \mathcal{P} \mathcal{L}^{-1}\right) \rho^{n+1} = \left(1 +  \frac{\sigma \Dt}{\eps^2} (1+\Dop)^{-1} \right) \mathcal{P}\mathcal{L}^{-1} f^n,
\end{eqnarray}
whose discrete version can now be solved by Krylov method such as GMRES. Preconditioned Krylov method has been applied to linear transport equation in various contexts, please refer to \cite{ABDH95, Brown95, GHP99} for an extensive study.

\subsection{Asymptotic preserving scheme via even-odd decomposition}
Another class of methods aim at capturing the diffusion limit as $\eps$ vanishes without resolving the mesh size. This is the spirit of so-called asymptotic preserving method. Various asymptotic preserving methods have been proposed for transport equation with diffusive scaling, and here we only recall the method proposed by Jin, Pareschi and Toscani \cite{JPT00} based on an even-odd decomposition as it motives our new method. Define an even and an odd part of $f$ by:
\begin{equation*}
f_\text{E} = \frac{1}{2}\left(f(t,x,\mu) +f(t,x,-\mu)\right)\,,\quad f_\text{O} = \frac{1}{2\eps}\left(f(t,x,\mu) - f(t,x,-\mu)\right)\,,
\end{equation*}
then (\ref{system: originalTransport}) splits into
\begin{equation} \label{JPT-1}
\begin{cases}
\partial_t f _\text{E}+  \mu\partial_xf_\text{O} = \frac{\sigma}{\varepsilon^2}(\rho - f_\text{E}),\\
\partial_t f_\text{O} + \alpha \mu\partial_xf_\text{E} = -\frac{\sigma}{\varepsilon^2}f_\text{O} + \frac{1}{\eps^2} (1-\alpha \eps^2 ) \mu \dx f_\text{E},
\end{cases}\,
\end{equation}
where $\alpha = \min\{1, \frac{1}{\eps^2}\}$. In this setting, two scales --- $\frac{1}{\eps}$ and $\frac{1}{\eps^2}$ --- in the original system has been unified to one scale in (\ref{JPT-1}) such that the non-stiff terms can now be treated explicitly while the stiff terms are treated implicitly. 

\vspace{0.3cm}

The rest of paper is organized as follows. In the next section, we explain at large the challenges in the implicit method and elucidate our idea---constructing a preconditioning matrix---by analyzing the spectral structure of the matrix that will be inverted. Detailed pseudo code is also provided. Section 3 is devoted to checking some properties of our scheme, especially the computational cost, asymptotic preservation and unconditionally stability. In Section 4 we present several numerical examples to illustrate the efficiency, accuracy, and AP properties of the schemes. Finally, the paper is concluded in Section 5.

\section{Numerical scheme}
Since our purpose is to capture both the diffusion and free streaming limit of (\ref{system: originalTransport}) (or \ref{transport_2D}) without the parabolic CFL constraint, we will take the fully implicit time discretization. However, as mentioned above, this amounts to invert a large scale matrix whose condition number is not necessarily acceptable. Below, we will first explain our main idea in
\begin{itemize}
\item utilizing the spectral structure of the collision operator to construct the preconditioner for the ill-conditioned matrix;
\item reducing the size of the matrix by separating the spatial dimension from the velocity dimension by building the sparsity structure into a matrix-free form.
\end{itemize}
Then a detailed algorithm is given in the next subsection. The idea is presented for the isotropic collision case and the anisotropic treatment will be demonstrated in Section 2.3.
  
\subsection{Challenges and ideas} \label{sec: ideas}
Recall the fully implicit semi-discretization of the equation in one dimension:
\begin{equation}\label{scheme: pen1}
\frac{f^{n+1}-f^n}{\Delta t} + \frac{\mu}{\varepsilon} \partial_xf^{n+1} = \frac{\sigma}{\varepsilon^2}(\rho^{n+1} - f^{n+1})\, ,
\end{equation}
where $f^n = f(t^n,x,\mu)$.
Here because of the convection, the stiff parts are asymmetric. To make it symmetric, we define an even-odd parity of the solution by:
\begin{equation*}
f_\text{E} = \frac{1}{2}\left(f(t,x,\mu) +f(t,x,-\mu)\right)\,,\quad f_\text{O} = \frac{1}{2}\left(f(t,x,\mu) - f(t,x,-\mu)\right)\,,
\end{equation*}
then the equation (\ref{scheme: pen1}) splits:
\begin{equation}
\begin{cases}
\frac{f^{n+1}_\text{E}-f^n_\text{E}}{\Delta t} + \frac{1}{\varepsilon} \mu\partial_xf^{n+1}_\text{O} = \frac{\sigma}{\varepsilon^2}(\rho^{n+1} - f^{n+1}_\text{E})\\
\frac{f^{n+1}_\text{O}-f^n_\text{O}}{\Delta t} + \frac{1}{\varepsilon} \mu\partial_xf^{n+1}_\text{E} = -\frac{\sigma}{\varepsilon^2}f^{n+1}_\text{O}
\end{cases}\,.
\end{equation}
From the second equation it is easy to get:
\begin{equation}\label{eqn:f_O_new}
f^{n+1}_\text{O} = \frac{\varepsilon^2}{\varepsilon^2+  \sigma \Delta t}\left(f^n_\text{O} - \frac{\Delta t}{\varepsilon} \mu \partial_xf^{n+1}_\text{E}\right)\,,
\end{equation}
and plug it back into the first equation we get the updating formula for $f^{n+1}_\text{E}$:
\begin{equation} \label{eqn:f_E_new}
- \dx \left( \frac{\eps^2 \mu^2 \Dt}{ \eps^2 + \sigma \Dt}  \dx f_\text{E}^{n+1}\right) 
+ \left( \frac{\eps^2}{\Dt} + \sigma \right) f_\text{E}^{n+1} - \sigma \rho^{n+1} 
= \frac{\eps^2}{\Dt} \left[ f_\text{E}^n - \dx \left( \frac{\eps \mu \Dt}{\eps^2+ \sigma \Dt} f_\text{O}^n\right) \right].
\end{equation}
Getting $f^{n+1}_\text{O}$ is easy once $f^{n+1}_\text{E}$ is known but solving~\eqref{eqn:f_E_new} is challenging. In a more compact form (\ref{eqn:f_E_new}) writes as:
\begin{equation}\label{eqn:ABinverse}
\left(\Aop+\Bop\right)f = b\,,
\end{equation}
with
\begin{equation}\label{eqn:ABb}
\Aop = -\frac{\varepsilon^2\Delta t}{\varepsilon^2+\Delta t}\mu^2\partial_x^2\,,
\quad \Bop = 1 +\frac{\varepsilon^2}{\Delta t} - \Pop\,,\quad\text{and}\quad b = \frac{\varepsilon^2}{\Delta t}\left(f^n_\text{E} - \frac{\varepsilon\Delta t}{\varepsilon^2+\Delta t} \mu \partial_xf^n_{O}\right)\,.
\end{equation}
Here we assume $\sigma \equiv 1$ for the moment to better illustrate the idea. 
We also used $\Pop f= \langle f\rangle$. Upon discretization, denote $\Amat$ and $\Bmat$ as the corresponding matrices for operators $\Aop$ and $\Bop$, then one needs to invert $\Amat+\Bmat$ for the solution and the difficulty is two folds:
\begin{itemize}
\item The size of the system. In what follows, we use an $N_v$-point quadrature and let $N_x$ be the number of spatial discretization points in space. Then $f$ is a vector of length $N_xN_v$ and matrices $\Amat$ and $\Bmat$ are of size $(N_x \times N_v)^2$. Inverting a matrix of such size is unrealistic especially in high dimensions if we think of, for example, dimension three, in which matrices reach a size of $(N_x^3 \times N_v^2)^2$. What is more, storing such matrices numerically requires memory that exceed what is affordable. However, the matrices enjoy some good properties that we could embed in the numerical schemes to save the memory and computational cost.
\begin{itemize}
\item Both $\Amat$ and $\Bmat$ are symmetric.
\item Both $\Amat$ and $\Bmat$ are positive definite. The eigenvalues of $\Bmat$ is either $1+\frac{\varepsilon^2}{\Delta t}$ or $\frac{\varepsilon^2}{\Delta t}$ and since $\Aop$ is an elliptic operator, the positivity of $\Amat$ is out of question.
\item  Both $\Amat$ and $\Bmat$ are block-wisely sparse. In particular, $\Amat$ is sparse in both coordinates while $\Bmat$ is dense in $\mu$ but sparse in $x$ coordinate. We can even write down its explicit expression:
\begin{equation} \label{eqn:matrices}
\Amat = \Amat_x\otimes \text{diag}\{\mu_k^2\}\,,\quad \Bmat = \Imat_{N_x}\otimes\Bmat_\mu\,.
\end{equation}
Here $\otimes$ denotes the Kronecker product, and we have used the notation $\Amat_x$ for the discretization of $\partial^2_x$, $\Bmat_\mu = \left(\Imat_{N_v}+\frac{\varepsilon^2}{\Delta t}-\Pmat\right)$ and the subindex for $\Imat$ indicate the size of this identity matrix. $\Pmat$ corresponds to the averaging operator $\Pop$ and is a dense matrix.
\end{itemize}

\item Ill-conditioning of the matrix. The matrix is well-conditioned in kinetic regime when $\varepsilon$ is relatively big ($\mathcal{O}(1)$), but in the diffusive regime when $\varepsilon\to 0$, roughly speaking the eigenvalues of $\Amat$ are very small at about $\varepsilon^2\mu^2k^2$ with $k=1\,,\cdots,N_x$, and thus the spectrum of the system is dominated by that of $\Bmat$. For a symmetric matrix, the condition number is the ratio between the biggest and the smallest eigenvalue, and thus it is $ \frac{1+\frac{\eps^2}{\Dt}}{\frac{\eps^2}{\Dt}} \sim \frac{\Delta t}{\varepsilon^2}\to\infty$ for fixed discretization in the diffusion limit. However, though $\Amat+\Bmat$ is ill-conditioned, the dominating part $\Bmat$ has very clear eigenspace structure. It shrinks constant vector by $\frac{\varepsilon^2}{\Delta t}<1$ in length and elongates all the vectors in the perpendicular space by $1+\frac{\varepsilon^2}{\Delta t}>1$. Specifically, we can write down the following eigendecomposition for $\Bmat_\mu$:
\begin{equation}\label{eqn:B_v_expansion}
\Bmat_\mu = \frac{\varepsilon^2}{\Delta t}ee^t + \left(1+\frac{\varepsilon^2}{\Delta t} \right)\sum_{i\geq 2}v_iv_i^t\,.
\end{equation}
Here we have used the notation $e = \frac{1}{\sqrt{N_v}}[1,\cdots,1]^t$ and $\RR^{N_v} = \text{span}\{e\,,v_i\,,\text{ for }i\geq 2\}$, where $e$ and all $v_i$ are orthonormal, meaning $\mathsf{X} = [e,v_2,\cdots,v_{N_v}]$ forms a unitary matrix.
\end{itemize}

\begin{remark}\label{remark: quadrature}
Note that the above properties for $\Amat$ and $\Bmat$ hold regardless of how we discretize the operator $\Aop$ or the quadrature of $\Bop$ except for the eigendecomposition (\ref{eqn:B_v_expansion}). However, only slight changes need to make to account for different quadratures. Denote $(\mu_1, \mu_2, \cdots \mu_{N_v})$ the quadrature points and $(w_1, w_2, \cdots w_{N_v})$ the corresponding weight, then $B_\mu$ can be decomposed as 
\begin{equation}\label{eqn:B_v_expansion2}
\Bmat_\mu = \frac{\varepsilon^2}{\Delta t} \sqrt{w}^{-1} v_1 v_1^t \sqrt{w} + \left(1+\frac{\varepsilon^2}{\Delta t} \right)\sum_{i\geq 2} \sqrt{w}^{-1} v_i v_i^t \sqrt{w}\, ,
\end{equation}
where $\sqrt{w} = \text{diag}( \sqrt{w_1}, \sqrt{w_2} , \cdots, \sqrt{w_{N_v}})$ and $\sqrt{w}^{-1}$ is the inverse of $\sqrt{w}$. Here $\RR^{N_v} = \text{span}\{ v_j \}$ for $1\leq j \leq N_v$ with $v_1 = [\sqrt{w_1}, \cdots, \sqrt{w_{N_v}}]^t$, and $v_j, ~ 1\leq j \leq N_v$ are orthogonal with respect to matrix $w = \text{diag}(w_1, \cdots w_{N_v})$, i.e., $v_i^t w v_j = \delta_{ij}$. Indeed, when we consider mid point rule on a uniform mesh, $w_i = \frac{1}{N_v}$ for all $i$ and (\ref{eqn:B_v_expansion2}) boils down to (\ref{eqn:B_v_expansion}).
\end{remark}

\begin{remark} 
The previous two numerical methods in the introduction partially used these good properties. The Krylov iteration method used the block-wisely sparse property of $\Amat$ and $\Bmat$ but moves the $\Pmat$ term to the right hand side, therefore suffering from slow convergence rate in the diffusion regime due to the canceling of $\Imat$ and $\Pmat$ in the spectrum. The even-odd parity decomposition explicitly uses the eigenvalue structure of $\Bmat$ and directly write the odd part as the $\mathcal{O}(\varepsilon)$ term, but since $\Amat$ term gets treated explicitly, the method could not overcome the parabolic scaling, meaning $\Delta t$ is controlled by $\Delta x^2$.
\end{remark}

\subsection{Numerical method} \label{subsection: numerical method}
Now, we design our algorithm by making full use of the properties of the two matrices presented above. Noting that $\Amat+\Bmat$ is positive definite, symmetric, and sparse, therefore conjugate gradient method is a fast solver for computing the inverse. But a pre-conditioner is called for to take care of the ill-conditioning. Usually once a pre-conditioner is involved, the sparsity would be lost, but here our pre-conditioner will not introduce us much trouble as it is treated in the matrix-free fashion and only a few number of flops are needed, as will be explained below.

Simply put, the pre-conditioner we use is $\Bmat^{-1}$, and then the system (\ref{eqn:ABinverse}) becomes:
\begin{equation} \label{system-AB}
(\Bmat^{-1}\Amat + \Imat ) f = \Bmat^{-1}b\,.
\end{equation}
It is easy to show that $\Bmat^{-1}\Amat + \Imat$ is symmetric positive definite, and thus conjugate gradient method can be used, which brings a need for fast matrix-vector multiplication solver. This is, however, not straightforward. Indeed, note that the matrix $\Bmat^{-1}$ is only block-wise sparse (much denser as compared to $\Bmat$), and if we use a naive multiplication, it will take part in a large portion of computation effort and degrades the advantage of using conjugate gradient method. To deal with this problem, we notice that $\Bmat_\mu^{-1}$ has a special structure so that it can be decomposed in such a way that mimics a low-rank decomposition. Specifically, let $g(x,\mu)$ be a function depends on both $x$ and $\mu$. Then upon discretization, $\Bmat_\mu^{-1} g$ can be directly expressed by, as shown in~\eqref{eqn:B_v_expansion}:
\begin{equation}\label{eqn:B_v_inv_expansion}
\Bmat^{-1}_\mu g = \frac{\Delta t}{\varepsilon^2}ee^t g + \frac{1}{1+\frac{\varepsilon^2}{\Delta t}}\left(g-ee^t g\right)\,
\end{equation}
for a fixed $x$, and then $\Bmat^{-1} g$ is to apply the equation above for every $x$. The multiplication with $\Amat$ is cheap due to sparsity. 

\begin{remark}
As explained in Remark \ref{remark: quadrature}, when we consider a different quadrature, (\ref{eqn:B_v_inv_expansion}) generalizes to  
\begin{equation} \label{eqn:B_v_inv_expansion2}
\Bmat^{-1}_\mu g = \frac{\Delta t}{\varepsilon^2} \sqrt{w}^{-1}v_1v_1^t \sqrt{w} g + \frac{1}{1+\frac{\varepsilon^2}{\Delta t}}\left(g-  \sqrt{w}^{-1}v_1 v_1^t \sqrt{w} g\right)\,
\end{equation}
where $v_1$ and $\sqrt{w}$ are defined in Remark \ref{remark: quadrature}.
\end{remark}

We summarize the algorithm below in Algorithm 1. 
\RestyleAlgo{boxruled}
\begin{algorithm}
\KwData{Initial data: $f(t=0,x,v)$, final time $T$}
\KwResult{$f(t_n,x,v)$ for all time steps $t_n$.}
Discretization: $\Delta t$, $\Delta v$, $\Delta x$\;
Initialization: $t=0$, $f_\text{E}(0,x,v) = \frac{1}{2}\left(f(0,x,v)+f(0,x,-v)\right)$, $f_\text{O} = f-f_\text{E}$\;
Matrix Preparation: $\Amat_x$, $e$\;
\While{$t<T$}{Set $b$ as in~\eqref{eqn:ABb}\; Compute $f_{E}$ using~\eqref{eqn:ABinverse}, call function {\bf ABinverse}\; Compute $f_\text{O}$ using~\eqref{eqn:f_O_new} directly\;Set $t = t+\Delta t$\;
}
\caption{Outline of the {\bf main} algorithm}\label{alg:main}
\end{algorithm}

\RestyleAlgo{boxruled}
\begin{algorithm}
\KwData{Initial data: $b$.}
\KwResult{$f$.}
Set $b_\text{mod} = \Bmat^{-1}_vb$ using~\eqref{eqn:B_v_expansion}\;
Compute $\Bmat^{-1}(\Amat+\Bmat)f = \Bmat^{-1}b$ using CG, in which we call function~{\bf ABproduct}\;
\caption{Function {\bf ABinverse}: solving $(\Amat+\Bmat)f = b$}\label{alg:ABinverse}
\end{algorithm}
\RestyleAlgo{boxruled}
\begin{algorithm}
\KwData{Initial data: $b$.}
\KwResult{$f = \Bmat^{-1}(\Amat + \Bmat)b$.}
Compute $b_1 = \Amat\cdot b$\;
Compute $b_2 = \Bmat^{-1}\cdot b_1$ using~\eqref{eqn:B_v_inv_expansion}\;
Compute $f = b_2 + b$\;
\caption{Function {\bf ABproduct}: computing $f = \Bmat^{-1}(\Amat + \Bmat)b$}\label{alg:ABproduct}
\end{algorithm}

Note that in Algorithm~\ref{alg:ABinverse}, CG is called which requires the matrix product of $\Bmat^{-1}(\Amat+\Bmat)$. But it is not applied directly. Instead, we applied function~{\bf ABproduct}. It is also worth mentioning that construction of $\Amat$ can be done with great generality as neither the computation of our pre-conditioner nor the fast matrix-vector multiplication technique here depends on the specific form of spatial discretization. In the examples we will show later, we use center difference for spatial derivatives and mid-point rule on a uniform gird for the averaging operator in $\mu$. But other quadrature rule such as Gaussian quadrature used in $S_N$ approximations \cite{AL02} and various spatial discretization such as discontinuous Galerkin method \cite{OHF12}  or discontinuous finite-element method \cite{WMMP01} can be easily adapted. Now, it is important to point out that our algorithm, majorly solving the preconditioned system (\ref{system-AB}), only requires a three simple matrix-vector multiplications, two with $\Bmat^{-1}$ and one with $\Amat^{-1}$, in each iteration, whereas in the previous method (\ref{pk-3}), quite a few operations are needed: a diffusion solver, a transport sweep and an average calculation, not to mentioned some of them are needed more than twice. Therefore the main benefit of our method is its simplicity.

\begin{remark} \label{remark:GMRES}
The above idea of constructing and computing the pre-conditioner can be applied directly to the original system (\ref{scheme: pen1}) without symmetrization technique via the even-odd parity. Indeed, rewrite (\ref{scheme: pen1}) in a compact form
\begin{equation}
(\Cop + \Bop) f^{n+1} = d
\end{equation}
where $\Cop = \eps \mu \dx$, $\Bop$ is the same as (\ref{eqn:ABb}), and $d = \frac{\eps^2}{\Dt} f^n$. Here we again assume $\sigma \equiv 1$ for simplicity. Then it boils down to invert the discretization matrix $\Cmat+ \Bmat$, which can be done similarly using the previous framework expect changing conjugate gradient method to Generalized Minimal Residual method (GMRES). 
\end{remark}

\begin{remark}\label{remark: 2nd}
Our scheme is easily extended to second order. As a center difference is used for the spatial derivative, we only need to upgrade our time discretization to second order. This is done by replacing (\ref{scheme: pen1}) with
\begin{eqnarray}
\frac{3f^{n+1}-4f^n+f^{n-1}}{2\Delta t} + \frac{1}{\eps} \mu\cdot \partial_x f^{n+1} = \frac{\sigma}{\eps^2} (\rho^{n+1} - f^{n+1}).
\end{eqnarray}
Then the rest steps are the same as that described in Section \ref{subsection: numerical method} except varying a few constants.
\end{remark}

\subsection{Anisotropic scattering}
In this subsection, we generalize the above framework to the anisotropic scattering. In one dimensional slab geometry, the transport equation takes the form
\begin{equation}
\dt f + \frac{\mu}{\eps} \dx f = \frac{\sigma_0(x)}{2\eps^2 } \int_{-1}^1 \sigma(\mu \cdot \mu') (f(\mu') -  f(\mu))  d\mu',
\end{equation}
where $\frac{1}{2}\int_{-1}^1 \sigma(\mu\cdot \mu') d\mu' = 1$. In this case, the diffusion limit is 
\begin{eqnarray} \label{diff-aniso}
\rho_t  + \frac{1}{2}\int_{-1}^{1}   \mu \partial_x   \left(  \mathcal{Q}^{-1} \left( \frac{\mu \partial_x \rho}{\sigma_0(x)}  \right)  \right) d\mu = 0,
\end{eqnarray}
where $\mathcal{Q}(f) = \frac{1}{2}\int_{-1}^1 \sigma(\mu\cdot \mu') f(\mu') d\mu' - f(\mu)$ is the linear collision operator. The main idea of the numerical scheme follows that in Section \ref{sec: ideas}, and here we present a non-symmetric version as mentioned in Remark \ref{remark:GMRES}, a reformulation to symmetric form is similar to (\ref{eqn:f_O_new})--(\ref{eqn:f_E_new}) and we leave it to the reader. 

Denote 
\begin{eqnarray}\label{Pop_sigma}
\Pop^\sigma f = \frac{1}{2}\int_{-1}^1 \sigma(\mu\cdot \mu') f(\mu') d\mu',
\end{eqnarray}
and assume $\sigma_0(x) \equiv 1$ for simplicity, then the semi-discrete scheme reads
\begin{eqnarray} \label{imp-aniso}
(\Bmat^\sigma + \Cmat) f^{n+1} = b.
\end{eqnarray}
Here $\Bmat^\sigma$ and $\Cmat$ are the discretization matrices for operator $1+\frac{\eps^2}{\Delta t} - \Pop^\sigma$ and $\eps \mu \partial_x$, respectively, and $d = \frac{\eps^2}{\Delta t} f^n$. Write $\Bmat^\sigma = \Imat_{N_x} \otimes \Bmat^\sigma_\mu$, and $\Bmat^\sigma_\mu = (1+\frac{\eps^2}{\Delta t})I_{N_v} - \Pmat^\sigma_\mu$, where $\Pmat^\sigma_\mu$ is the corresponding matrix for $\Pmat^\sigma$ with fixed $x$. Then the spectral structure of $\Bmat^\sigma_\mu$ is similar to $\Bmat_\mu$ in (\ref{eqn:B_v_expansion}) and we summarize as follows:
\begin{itemize}
\item The matrix $\Pmat^\sigma_\mu$ is low rank as it depends on the inner product of $\mu$ and $\mu'$. Denote its eigenvalues $\xi_1>  \xi_2  > \cdots > \xi_k$, $k\ll N_v$, and corresponding eigenvectors $v_1, \cdots, v_k$.
\item The largest eigenvalue of $\Pmat^\sigma_\mu$ is $\xi_1 = 1$, and corresponding eigenvector is $e = \frac{1}{\sqrt{N_v}}[1,\cdots, 1]^t$.
\item Denote the eigenvalues of $\Bmat_\sigma^\mu$ as $\lambda_1 \leq \lambda_2 \cdots \leq \lambda_{N_v}$. Then $\lambda_1 =  \frac{\eps^2}{\Delta t}$ and $\lambda_{k+1} = \lambda_{k+2} = \cdots \lambda_{N_v} =  1+\frac{\eps^2}{\Delta t}$. The eigenvector corresponding to $\lambda_1$ is $e$. Other eigenvalues depend on the form of $\sigma(\mu\cdot\mu')$.
\item The eigenvectors $v_j$, $j=k+1, \cdots N_v$ corresponding to the eigenvalues $\lambda_j = 1+\frac{\eps^2}{\Delta t}$ all satisfy $\Pmat_\sigma^\mu v_j = 0$.
\end{itemize}
Therefore, for any vector $g$ of size $N_v$, first find its projection to the first $k$ eigenvectors $v_1$, $\cdots$, $v_k$. That is, write $g = \sum_{i=1}^{N_v} c_i v_i$, and find $c_1$, $\cdots$, $c_k$. Then, we have
\begin{eqnarray} \label{B-aniso}
\Bmat^\sigma_\mu g = \sum_{i=1}^k c_k \lambda_k v_k + \left(1+\frac{\eps^2}{\Delta t}\right) (g-\sum_{i=1}^k c_k v_k),
\end{eqnarray} 
and 
\begin{eqnarray}\label{B-inv-aniso}
(\Bmat^\sigma_\mu)^{-1} g =\sum_{i=1}^k c_k \lambda_k^{-1} v_k + \frac{1}{1+\frac{\eps^2}{\Delta t}} (g-\sum_{i=1}^k c_k v_k).
\end{eqnarray} 
Since $k \ll N_v$, the computation (\ref{B-aniso}) and (\ref{B-inv-aniso}) are cheap. In fact, in most applications, the scattering $\sigma(\mu\cdot \mu')$ has special structures such that $v_2$, $\cdots$, $v_k$ are easy to find out. Then what is left is to precondition (\ref{imp-aniso}) as
\begin{eqnarray}
(\Bmat^\sigma)^{-1}(\Bmat^\sigma + \Cmat) f^{n+1} = b,
\end{eqnarray}
and solve the resulting linear system using GMRES.

\section{Properties}
We study the numerical properties of our scheme in this section. We will concentrate on the equation with isotropic collision operator and the extension to the anisotropic case is straightforward, which we will omit from here. 

\subsection{Conditioning of $\Bmat^{-1}(\Amat + \Bmat)$}
In the isotropic scattering case, it is easily seen that $\Bmat^{-1}$ only has two eigenvalues, $\frac{\Delta t}{\varepsilon^2}$ and $\frac{\Delta t}{\Delta t+\varepsilon^2}$. Considering $\Amat$ is the discretization of $-\frac{\varepsilon^2\Delta t}{\varepsilon^2 +\Delta t}\mu^2\partial_x^2$, the spectrum is roughly given by (depending on the boundary too):
\begin{equation}
-\frac{\varepsilon^2\Delta t}{\varepsilon^2 +\Delta t}\mu^2 k^2\,,\quad\text{with }k = 0, 1\,,\cdots N_x-1\,,
\end{equation}
and thus the smallest possible and the biggest possible eigenvalues of $\Bmat^{-1}\Amat + \Bmat$ are:
\begin{itemize}
\item{biggest:} $\frac{\Delta t}{\varepsilon^2}\cdot\frac{\varepsilon^2\Delta t^2}{\varepsilon^2+\Delta t^2}(N_x-1)^2$,
\item{smallest:} 0
\end{itemize}
The condition number of $\Bmat^{-1}(\Amat + \Bmat)$, therefore would be:
\begin{equation}
\kappa = \frac{1+\frac{\Delta t}{\varepsilon^2}\cdot\frac{\varepsilon^2\Delta t^2}{\varepsilon^2+\Delta t^2}N_x^2}{1+\frac{\Delta t}{\Delta t + \varepsilon^2}\cdot\frac{\varepsilon^2\Delta t^2}{\varepsilon^2+\Delta t^2}}\sim 1+\frac{1}{\Delta t}\,.
\end{equation}
Here we are in the regime of $\varepsilon\ll\Delta t$. The condition number would not produce much trouble in kinetic regime where $\varepsilon$ is a moderate number.

\subsection{Computational cost}
We first analyze the computational cost per each time step. Assume that we are given $f^n_\text{E}$ and $f^n_\text{O}$, we calculate the flops it takes to update for $f^{n+1}_\text{E}$ and $f^{n+1}_\text{O}$. As can be seen in Algorithm~\ref{alg:main}, $b$ is computed using ~\eqref{eqn:ABb} and $f_\text{O}$ is computed using~\eqref{eqn:f_O_new}, both of which require $\mathcal{O}(N_x)$ flops of computation per angular grid and thus $\mathcal{O}(N_xN_v)$ flops of computation in total. To compute $f^{n+1}_\text{E}$, the scheme calls Algorithm~\ref{alg:ABinverse}, in which $b_\text{mod}$ requires $N_xN_v$ flops, and conjugate gradient is used for $f^{n+1}_\text{E}$. Denote $\text{Tol}$ the tolerance set for the method, the number of iteration needed for CG is:
\begin{equation}
N_\text{iter}\geq \frac{\log{\text{Tol}}}{\log\left(\frac{\sqrt{\kappa}-1}{\sqrt{\kappa}+1}\right)}
\end{equation}
where $\kappa$ is the condition number for matrix $\Bmat^{-1}(\Amat + \Bmat)$. With the results above, the matrix is well-conditioned with $\kappa\sim 1$. and thus the required iteration is small. For example, if we set $\text{Tol} = 1e-10$ and $\Delta t = 1e-3$, then:
\begin{equation}
N_\text{iter}\gtrsim \frac{\log{\text{Tol}}}{\log\left(\frac{\frac{1}{2}\Delta t^3}{2+\frac{1}{2}\Delta t^3}\right)} \gtrsim \frac{\log{\text{Tol}}}{3\log \Delta t }=\frac{10}{9}\,.
\end{equation}
In each iteration in the conjugate gradient, the matrix multiplication of $\Bmat^{-1}(\Amat+\Bmat)$ needs to be applied three times. Since our scheme is matrix free (as shown in Algorithm~\ref{alg:ABproduct}), meaning we perform $\Amat$ product first, which requires $\mathcal{O}(N_xN_v)$ flops, and then perform $\Bmat^{-1}$ product, which also requires the same amount of flops, in total, each iteration in conjugate gradient requires $\mathcal{O}(N_xN_v)$ flops. Multiplied by $\mathcal{O}(1)$ iterations required by the CG method, $f^{n+1}_\text{E}$  is computed using $\mathcal{O}(N_xN_v)$ flops.

\subsection{Asymptotic preservation}
Now we check the asymptotic preservation property for the scheme~\eqref{scheme: pen1}. First taking the average of (\ref{scheme: pen1}), one has
\begin{eqnarray}\label{AP-1}
\frac{\rho^{n+1} - \rho^n}{\Dt} + \frac{1}{\eps}\average{ \mu \dx f^{n+1}} = 0.
\end{eqnarray}
Notice that 
\begin{eqnarray*}
f^{n+1} = \rho^{n+1} - \frac{\eps \mu}{\sigma} \dx f^{n+1} - \frac{\eps^2}{\sigma \Dt} (f^{n+1}-f^n).
\end{eqnarray*}
Plugging it into (\ref{AP-1}) leads to
\begin{equation} \label{AP-2}
\frac{\rho^{n+1} - \rho^n}{\Dt} +  \dx \left( \frac{1}{3\sigma} \dx \rho^{n+1} \right) = \mathcal{O}(\eps),
\end{equation}
which, upon sending $\eps$ to zero, is a semi-discrete implicit solver for the diffusion limit.

\subsection{Stability}
Next, we prove the stability of our scheme (\ref{scheme: pen1}). The result is summarized as follows. 
\begin{theorem}
The scheme (\ref{scheme: pen1}) is unconditionally stable. 
\end{theorem}
\begin{proof}
Multiply the equation (\ref{scheme: pen1}) by $f^{n+1}$, and integrate with respect to $x$ and $v$, we have
\begin{eqnarray} \label{stability-1}
\frac{1}{2} \Big[ \Average{(f^{n+1})^2} - \Average{(f^{n})^2} +  \Average{(f^{n+1}-f^n)^2}   \Big] =\frac{ \sigma \Dt}{\eps^2} \Big( \Average{ \rho^{n+1} f^{n+1} }- \Average{(f^{n+1})^2}  \Big),
\end{eqnarray}
Here we use $\Average{\cdot}$ to denote the integration with respect to both $x$ and $v$ and the equality $a(a-b) = \frac{1}{2} \left[ a^2 - b^2 + (a-b)^2 \right]$. Considering $\Average{ \rho^{n+1} f^{n+1} } = \tnorm{(\rho^{n+1})^2}$, with $\tnorm{\cdot}$ denoting integration only on $x$, and by Cauchy-Schwartz inequality, we have:
\begin{equation}
\Average{ \rho^{n+1} f^{n+1}}\leq\Average{(f^{n+1})^2}\,,
\end{equation}
making the right hand side negative. Therefore
\begin{equation}
\Average{(f^{n+1})^2}\leq\Average{(f^n)^2}\,,
\end{equation}
regardless of the choice of $\Dt$ and thus unconditional stability is guaranteed. 
\end{proof}

\section{Numerical examples}
In this section, we present several numerical examples using our AP scheme described in Algorithm 1--3. Here we use uniform mesh for both spatial and velocity discretization. Periodic boundary condition is adopted in both directions.

\subsection{Condition number comparison}
We check the effectiveness of the preconditioner by comparing the condition number of $\Amat+\Bmat$ and $\Bmat^{-1}\Amat + \Imat$ with different $\varepsilon$. It can be seen from Table~\ref{table:cond_ep0_pre} that with moderate $\varepsilon$, the condition number of $\Amat + \Bmat$ is relatively small. However, as $\varepsilon$ shrinks to zero, the condition number grows drastically, as displayed in the left three columns of Table~\ref{table:cond_ep5_pre}. This condition number is adequately reduced by the pre-conditioner $\Bmat^{-1}$, which is represented in the last three columns of Table~\ref{table:cond_ep5_pre}.
\begin{table}
\centering
 \begin{tabular}[ht]{|l|c|c|c|}
\hline
\multicolumn{4}{|c|}{Condition number for $\Amat + \Bmat$ with $\varepsilon = 1$.} \\
\hline
& $N_v = 10$ & $N_v = 20$ & $N_v = 30$\\
$N_x = 20$ & 1.34748186880472 & 1.38654551078317 & 1.40000848404069\\
$N_x = 40$ & 1.35413818003715 & 1.39425840610708 & 1.40810090542130\\
$N_x = 60$ & 1.35618972627881 & 1.39664857130697 & 1.41061251454211\\
$N_x = 80$ & 1.35718028325749 & 1.39780546942190 & 1.41182902863886\\
$N_x = 100$ & 1.35776282383554 & 1.39848680698068 & 1.41254576051502\\
\hline
 \end{tabular}
\caption{Condition number for $\Amat + \Bmat$ with $\varepsilon = 1$.}\label{table:cond_ep0_pre}
 \end{table}

\begin{table}
\centering
\begin{tabular}[ht]{|l|c|c|c|}
\hline
\multicolumn{4}{|c|}{Condition number for $\Amat + \Bmat$.} \\
\hline
& $N_v = 10$ & $N_v = 20$ & $N_v = 30$\\
$N_x = 20$ &1.59e8 & 1.59e8 & 1.59e8\\
$N_x = 40$ &8.12e7 & 8.12e7 & 8.12e7\\
$N_x = 60$ &5.46e7 & 5.46e7 & 5.46e7\\
$N_x = 80$ &4.11e7 & 4.11e7 & 4.11e7\\
$N_x = 100$ & 3.30e7 & 3.30e7 & 3.30e7\\
\hline
\end{tabular}
\begin{tabular}[ht]{|c|c|c|}
\hline
\multicolumn{3}{|c|}{Condition number for $\Bmat^{-1}\Amat + \Imat$.} \\
\hline
$N_v = 10$ & $N_v = 20$ & $N_v = 30$\\
15.88 & 16.14 & 16.19\\
31.50 & 32.04 & 32.14\\
47.12 & 47.94 & 48.09\\
62.74 & 63.84 & 64.05\\
78.37 & 79.75 & 80.00\\
\hline
\end{tabular}
\caption{Condition number for $\Amat + \Bmat$ and $\Bmat^{-1}\Amat + \Imat$ with $\varepsilon = 10^{-5}$.}\label{table:cond_ep5_pre}
\end{table}

\subsection{Computational cost comparison}
In this subsection we compare the computational cost for updating $f_\text{E}$ by only one time step. We update $f_\text{E}$ by both directly inverting the matrix and using CG method. In Table~\ref{table:cost_ratio} we show the ratio of the time used. With $N_x = 200$ and $N_v=20$, the time it takes by directly inverting the matrix is $130$ more than utilizing the CG method. Note that here we only report the saving from the online computation, assuming the computer could store the matrix. In practice, one also needs to consider the memory cost for directly inverting the matrix. The new AP solver, on the other hand, is a matrix-free method, and there is no need to store matrices.
\begin{table}
 \centering
 \begin{tabular}[ht]{|c|c|c|c|c|}
  \hline
  \multicolumn{4}{|c|}{Computational Cost $t_\text{inverse}/t_\text{cg}$.}\\
  \hline
&$\varepsilon = 1$ & $\varepsilon = 0.1$ & $\varepsilon = 0.01$\\
 $N_x = 50$ & 1.5733  &  2.6637  &  3.1782\\
    $N_x = 100$ & 54.3670  &  3.0818 &   4.4672\\
    $N_x = 150$ & 38.2581 &  13.1954 &   6.2918\\
   $N_x = 200$ & 130.6267 &  12.1222 &  10.9596\\
  \hline
 \end{tabular}
\caption{Computational cost ratio between directly inverting the matrix, and utilizing the conjugate gradient method.}\label{table:cost_ratio}
\end{table}

\subsection{One dimensional problem with isotropic scattering}
In this section, we test the efficiency of our numerical schemes in slab geometry. Here $x\in [0,2]$. In all examples, $\Dt $ is chosen to be $\Dx/3$ when using our scheme. 
\begin{itemize}
\item Example I:  
We consider the following initial condition 
\begin{eqnarray} \label{ic}
f(0,x, \mu) = \left\{ \begin{array}{ll}
2 &  0.8<x<1.2
\\ 0 &  \textrm{otherwise, }
\end{array}\right.
\end{eqnarray}
and isotropic scattering with vanishing cross section (see the left figure of Fig.~\ref{fig: cross-section})
\begin{equation} \label{ex1-xsection}
\sigma(x) = 100(x-1)^4.
\end{equation} 
For $\eps = 1$, we compare the solution using our scheme with the solution from an explicit solver. And the profile of $\rho$ at time $t_\textrm{max} = 1$ is displayed in the left figure of Fig.~\ref{fig: 1D1}. Here we choose $N_x = 200$ and $N_v = 100$. In the diffusion scaling, we choose $\eps = 10^{-3}$ and compare our solution with the solution of the diffusion limit (\ref{eqn:average}) at time $t_\textrm{max} = 0.1$, as shown in the right figure of Fig.~~\ref{fig: 1D1}. Good agreement are observed for both kinetic and diffusion regimes. Moreover, near the center of the computational domain where the scattering effect is extremely weak, our scheme successfully pick up the correct profile of density in the free streaming limit. 

\begin{figure}[!h]
\includegraphics[width=0.46\textwidth]{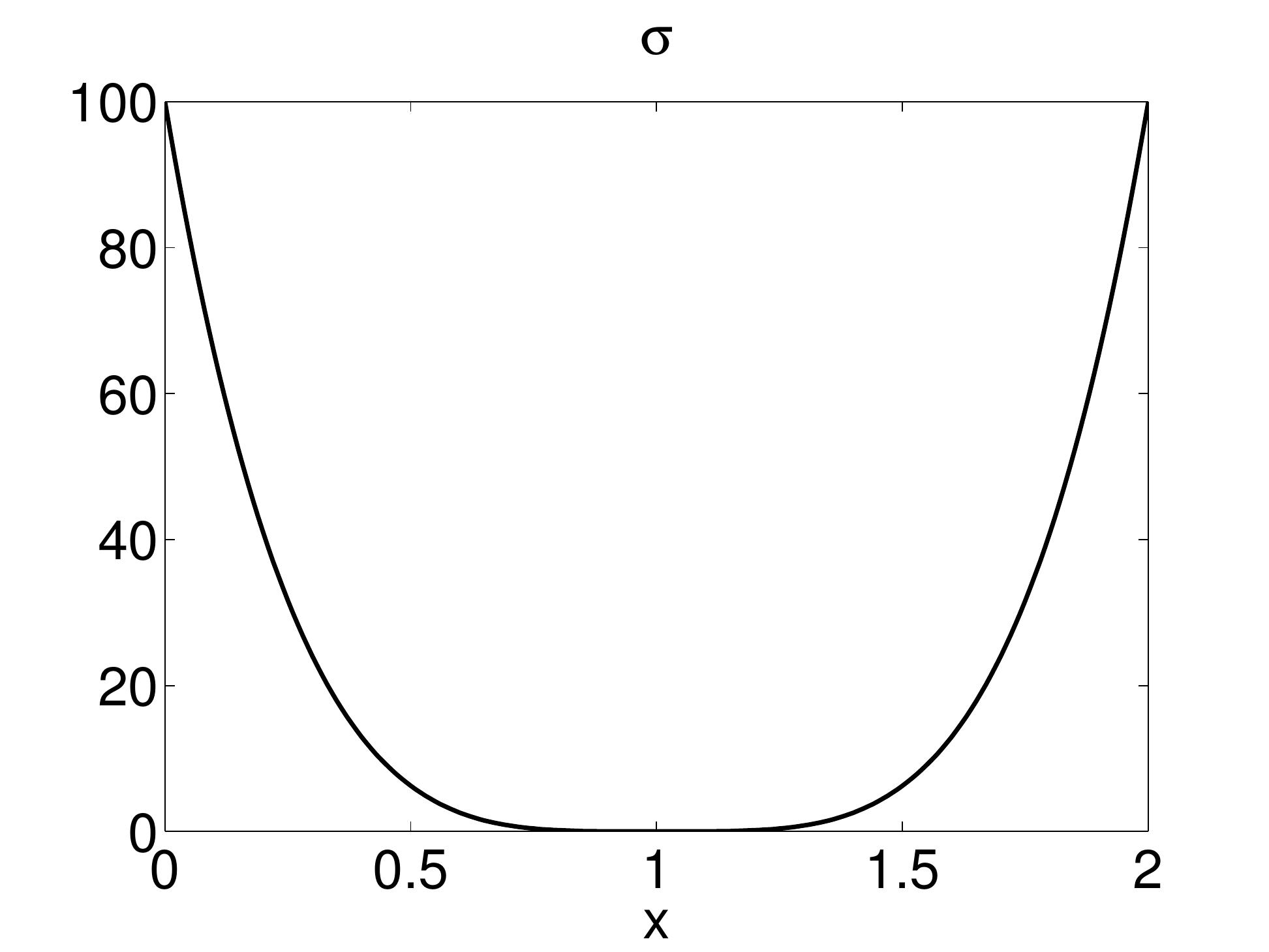}
\includegraphics[width=0.46\textwidth]{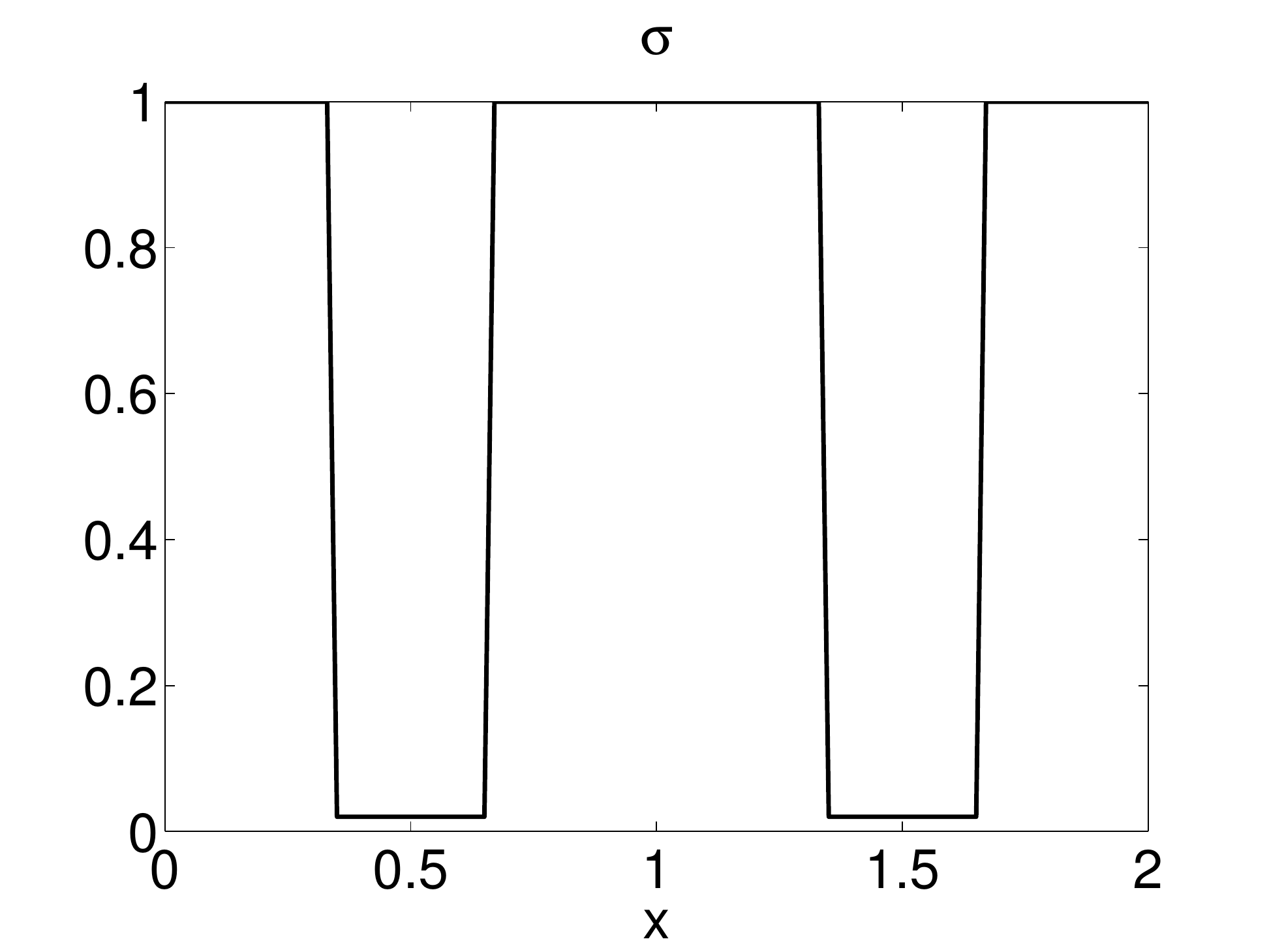}
\caption{Scattering cross section $\sigma$ for one dimensional examples. Left: vanishing cross-section (\ref{ex1-xsection}). Right: striped cross-section (\ref{ex2-xsection}). } \label{fig: cross-section}
\end{figure}

\begin{figure}[!h]
\includegraphics[width=0.46\textwidth]{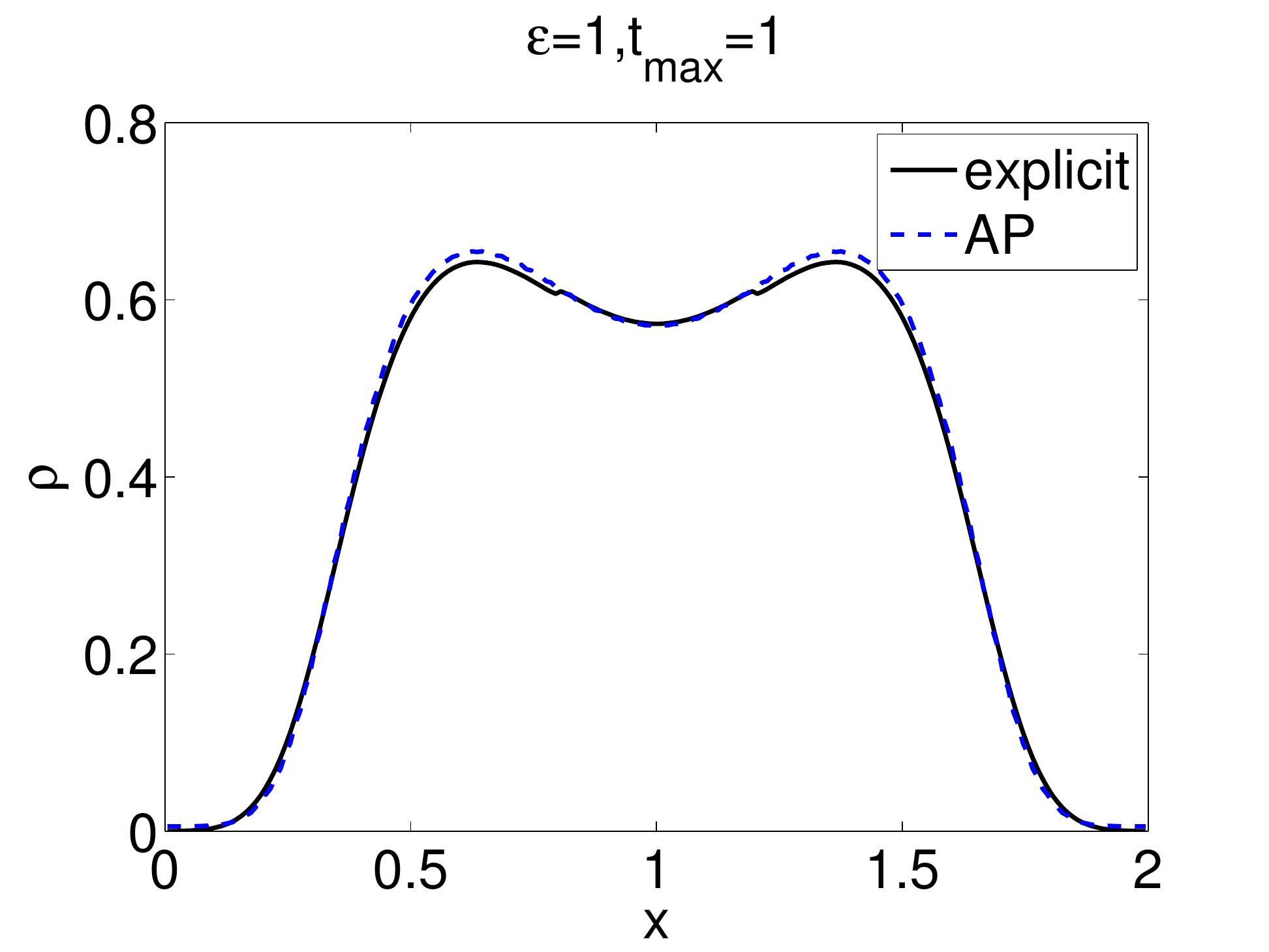}
\includegraphics[width=0.46\textwidth]{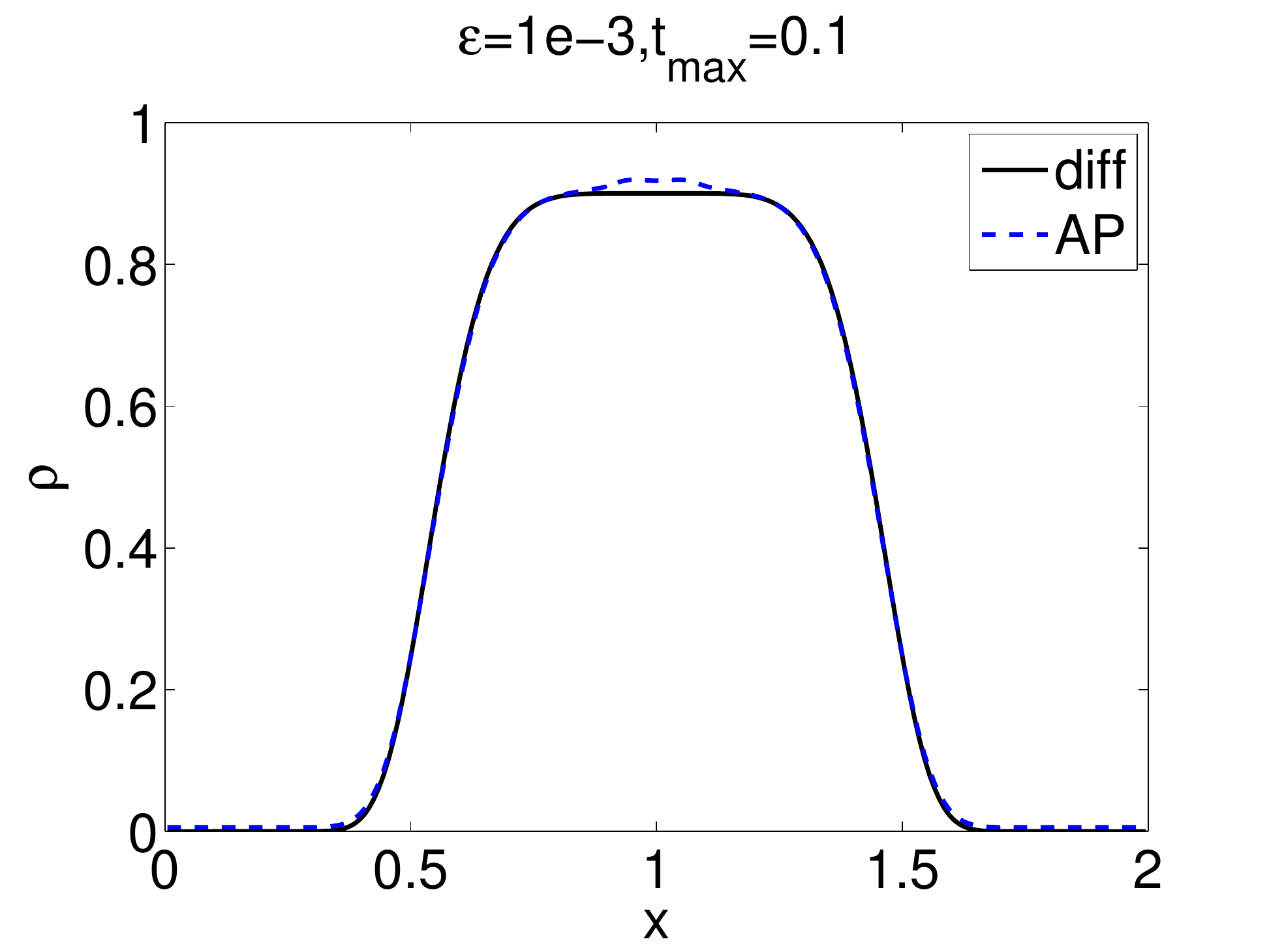}
\caption{Example I. Here $N_x = 200$, $N_v = 100$. Left: $\eps = 1$, we compare the density $\rho$ using our scheme (blue dashed curve) with the solution using explicit solver (black solid curve) at $t_\text{max} = 1$. Right: $\eps = 10^{-3}$, we compare the density $\rho$ using our scheme (blue dashed curve) with the solution to the diffusion limit (black solid curve) at $t_\text{max} = 0.1$.} \label{fig: 1D1}
\end{figure}

\item Example II: Initial condition is the same as in (\ref{ic}) and we consider striped cross section 
\begin{eqnarray} \label{ex2-xsection}
\sigma(x) = \left\{ \begin{array}{cc} 0.02 \quad &x\in [0.35, 0.65] \cup [1.35, 1.65] 
\\ 1 \quad &x \in [0, 0.35) \cup (0.65, 1.35) \cup(1.65,1] 
\end{array} \right.
\end{eqnarray}
so that the particles will transport from high scattering regime to low scattering regime and vice versa. 
Again we compute our solution for both $\eps = 1$ and $\eps = 10^{-3}$. In the former case, the solution is compared with the one using explicit solver, while in the later case, it is compared with the one using diffusion solver. The results are gathered in Fig. ~\ref{fig: 1D2}. This example validates the efficiency of our scheme in computing the transport equation with discontinuous cross-section, which is often the case in many real materials. 
 \begin{figure}[!h] 
\includegraphics[width=0.46\textwidth]{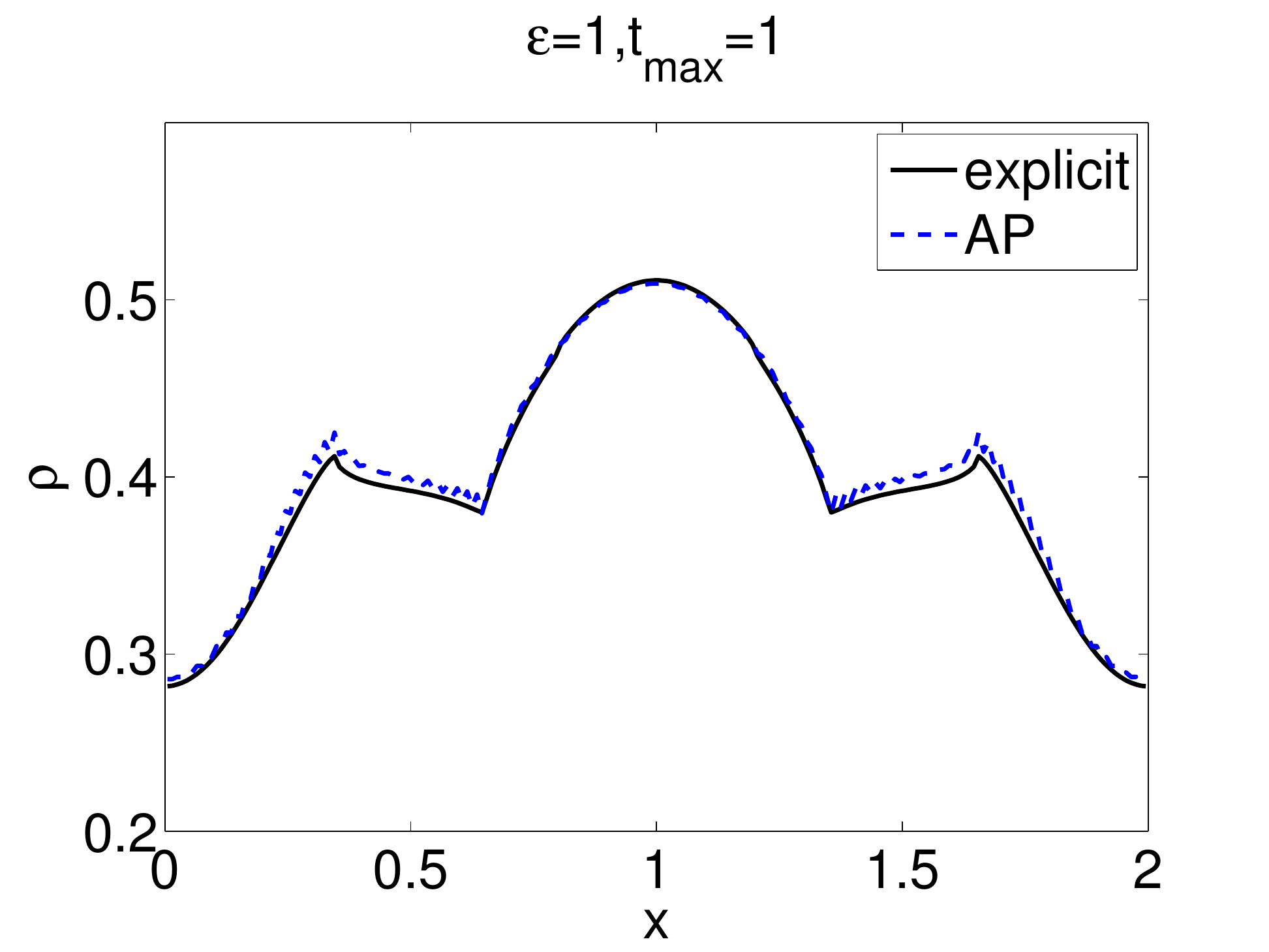}
\includegraphics[width=0.46\textwidth]{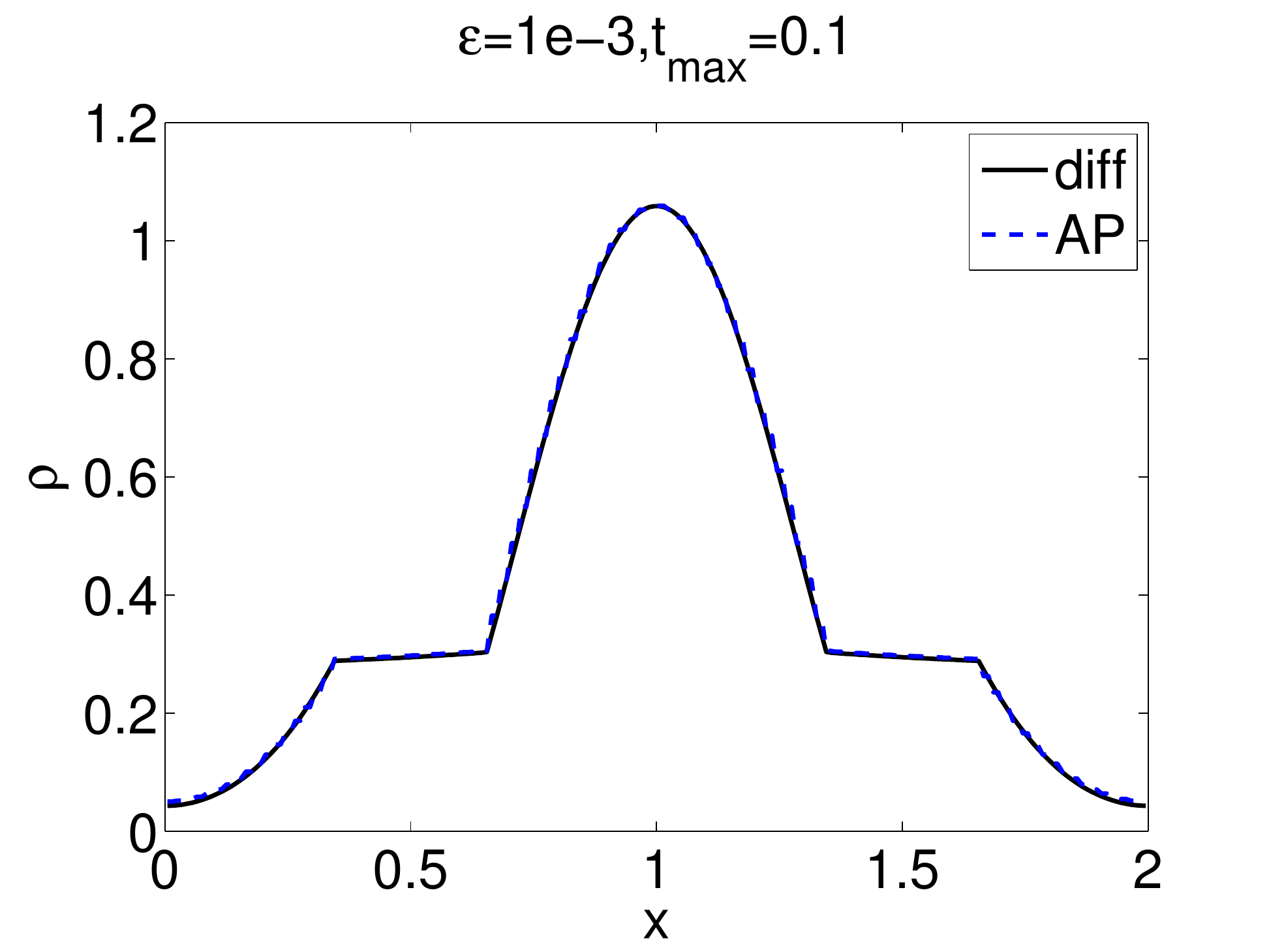}
\caption{Example II. Here $N_x = 200$, $N_v = 100$. Left: $\eps = 1$, we compare the density $\rho$ using our scheme (blue dashed curve) with the solution using explicit solver (black solid curve) at $t_\text{max} = 1$. Right: $\eps = 10^{-3}$, we compare the density $\rho$ using our scheme (blue dashed curve) with the solution to the diffusion limit (black solid curve) at $t_\text{max} = 0.1$. }\label{fig: 1D2}
\end{figure}

\item Example III:  Here we consider anisotropic scattering cross section 
\begin{eqnarray}
\sigma(x,\mu,\mu') = \sigma_0(x) \left( 1+  \mu \cdot \mu' \right)
\end{eqnarray}
with $\sigma_0(x)$ taking the form of (\ref{ex2-xsection}). For this special choice of $\sigma$, the diffusion limit (\ref{diff-aniso}) reduces to 
\begin{equation}
\rho_t + \frac{1}{2} \partial_x  \left( \frac{1}{\sigma_0} \partial_x \rho \right) = 0.
\end{equation}
Also, the eigenvalues of $\Bmat^\sigma_\mu$ are $\lambda_1 = \frac{\eps^2}{\Delta t}$, $\lambda_2 = 1+\frac{\eps^2}{\Delta t}-\frac{1}{3}$, and corresponding normalized eigenvectors are $v_1 = e$ and $v_2 = \frac{\mu}{\sqrt{\mu^t\mu}}$. And the rest eigenvalues are all equal to $ 1 + \frac{\eps^2}{\Delta t}$. Then for any vector $g = \sum_{i=1}^{N_v} c_i v_i$, notice that $\Pmat^\sigma_\mu g = c_1 v_1 + \frac{1}{3}c_2 v_2$, and $v_1$ is orthogonal to $v_2$, we have, $c_1 = v_1^t \Pmat^\sigma_\mu g$ and $c_2 = 3 v_2^T \Pmat^\sigma_\mu g$. Since $\Pmat^\sigma_\mu$ is low rank, the matrix-vector multiplication $\Pmat^\sigma_\mu g$ is cheap. Once $c_1$ and $c_2$ are computed, we can use the formula (\ref{B-inv-aniso}) with $k=2$ to compute the inverse of $\Bmat^\sigma_\mu$ and the rest steps are the same as isotropic ones. 

To illustrate, Fig. \ref{ex-3} on the left compares our solution with the solution using explicit scheme, and on the right compares with the solution to the diffusion limit. Here $\Delta x = 0.01$ and $\Delta v = 0.02$. 

\begin{figure} 
\includegraphics[width=0.46\textwidth]{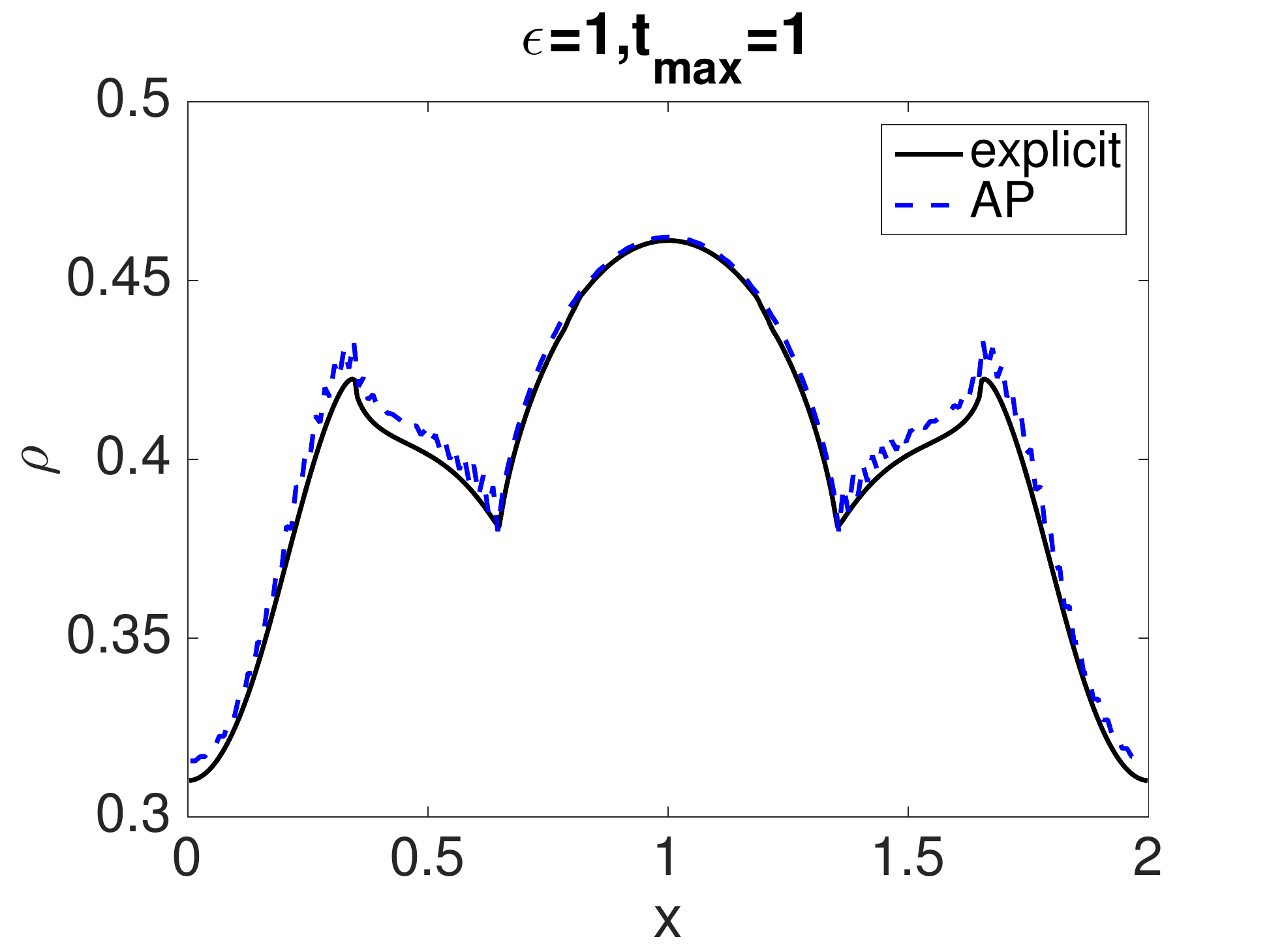}
\includegraphics[width=0.46\textwidth]{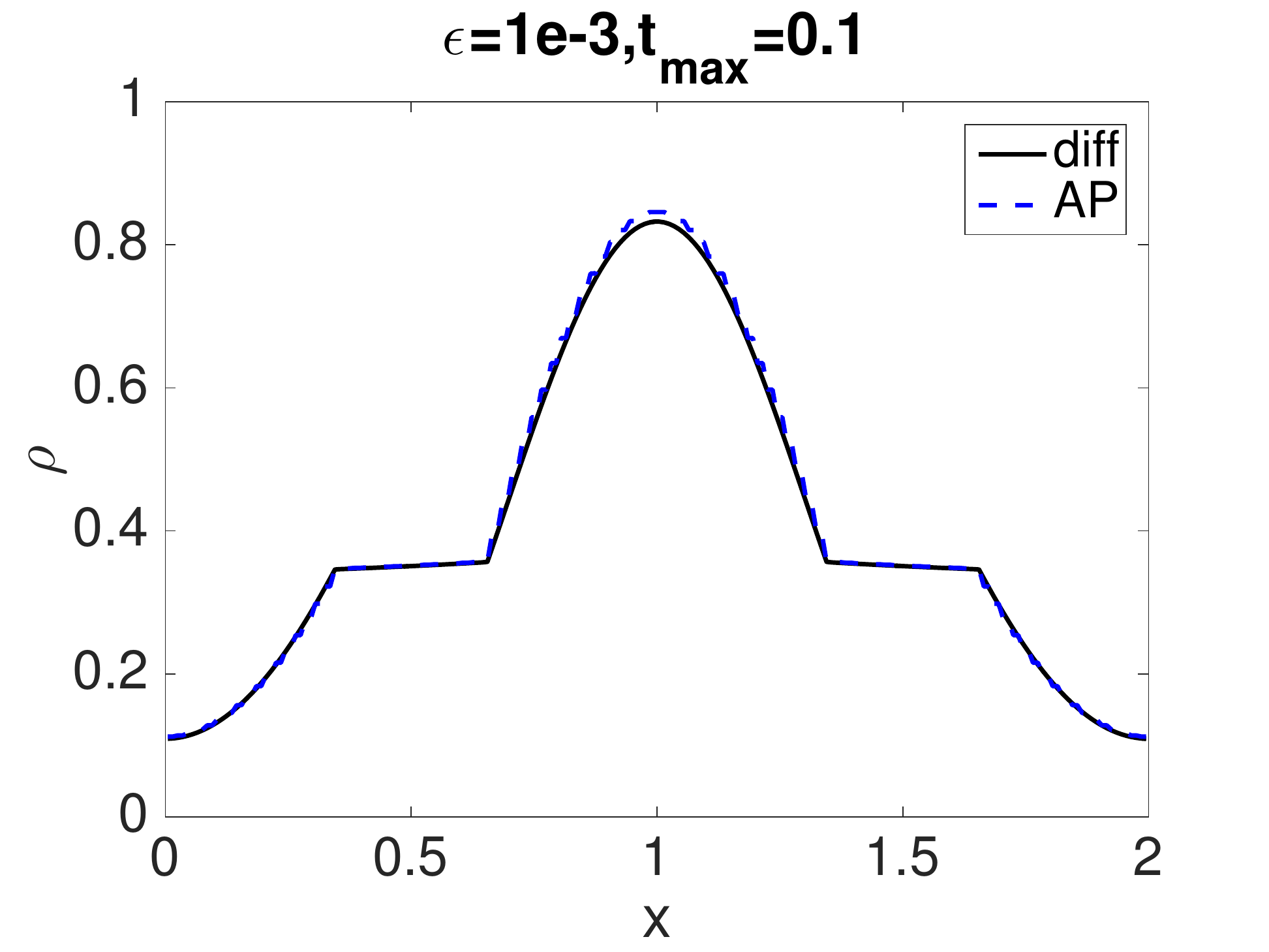}
\caption{Example III: anisotropic case. Here $N_x = 200$, $N_v = 100$. Left: $\eps = 1$, we compare the density $\rho$ using our scheme (blue dashed curve) with the solution using explicit solver (black solid curve) at $t_\text{max} = 1$. Right: $\eps = 10^{-3}$, we compare the density $\rho$ using our scheme (blue dashed curve) with the solution to the diffusion limit (black solid curve) at $t_\text{max} = 0.1$. } \label{ex-3}
\end{figure}

\end{itemize}

\subsection{Two-dimensional problems}
Finally we test our scheme in planar geometry. 
\begin{itemize}
\item Example V: Consider smooth initial condition 
\begin{equation}\label{ic-2D}
f(0,x,y,\xi,\eta) = 1+ e^{-40(x-0.5)^2 - 40(y-0.5)^2}, \quad 0 \leq x, y \leq 1
\end{equation}
and uniform cross section $\sigma(x,y) \equiv 1$. Here we check the asymptotic property of our scheme by computing the $l^2$ distance between $f$ and $\rho$, namely 
\begin{equation} \label{AP error}
|f-\rho|_2 = \sqrt{ \sum_i \sum_j |f(x_i,v_j) - \rho(x_i)|^2 \Dx \Delta \mu}
\end{equation}
with various $\eps$ along time, and the results are collected in Fig.~\ref{fig: AP error} left. As expected, this error decreases with $\eps$. And Fig.~\ref{fig: AP error} on the right further confirms the asymptotic property by comparing the density using our kinetic solver with that using a diffusion solver
\begin{equation} \label{AP error2}
|\rho_\text{k} - \rho_\text{d}|_2 = \sqrt{ \sum_i  | \rho_\text{k}(x_i) - \rho_\text{d}(x_i)|^2 \Dx }.
\end{equation}

\begin{figure}[!h] 
	\includegraphics[width=0.46\textwidth]{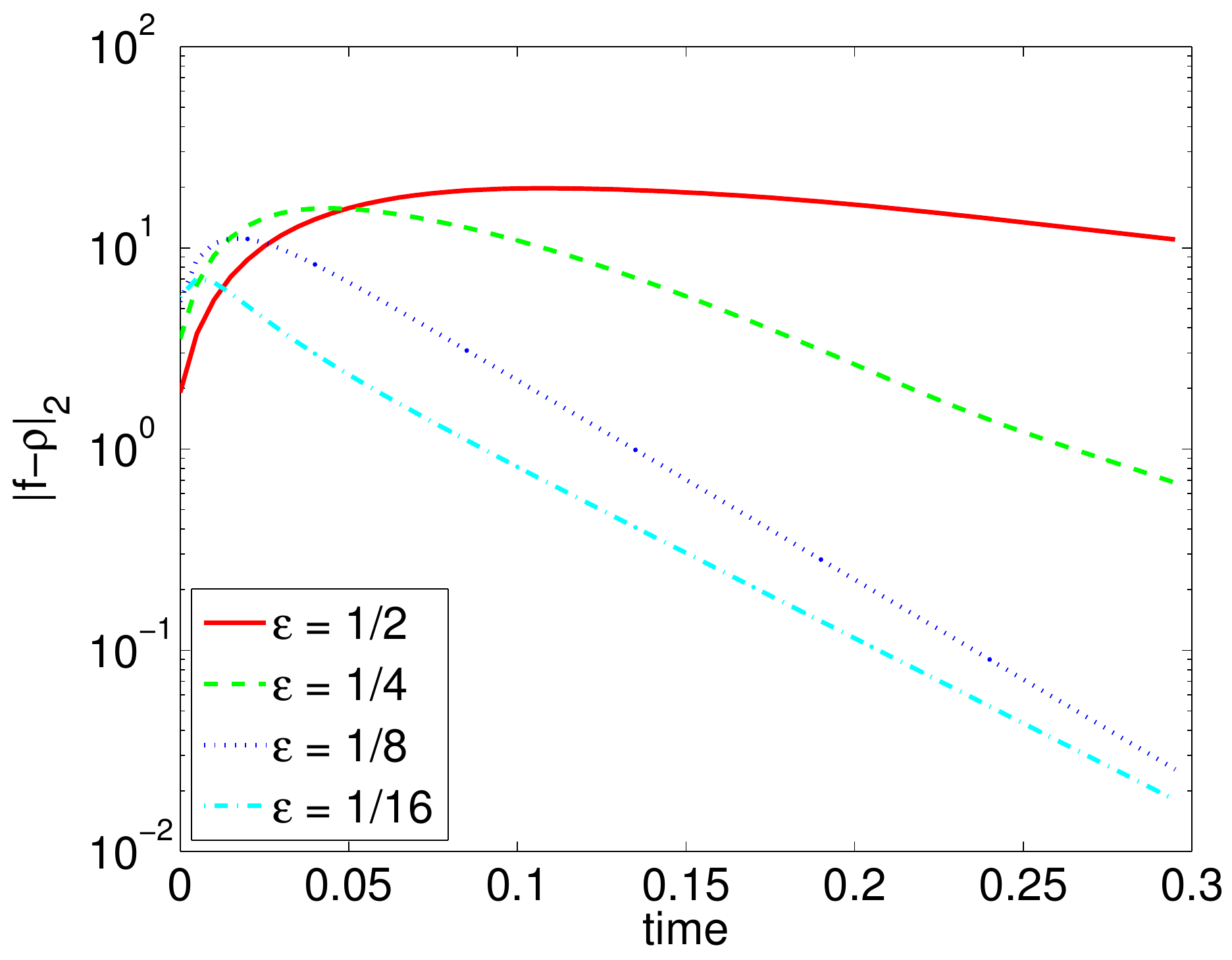}
	\includegraphics[width=0.45\textwidth]{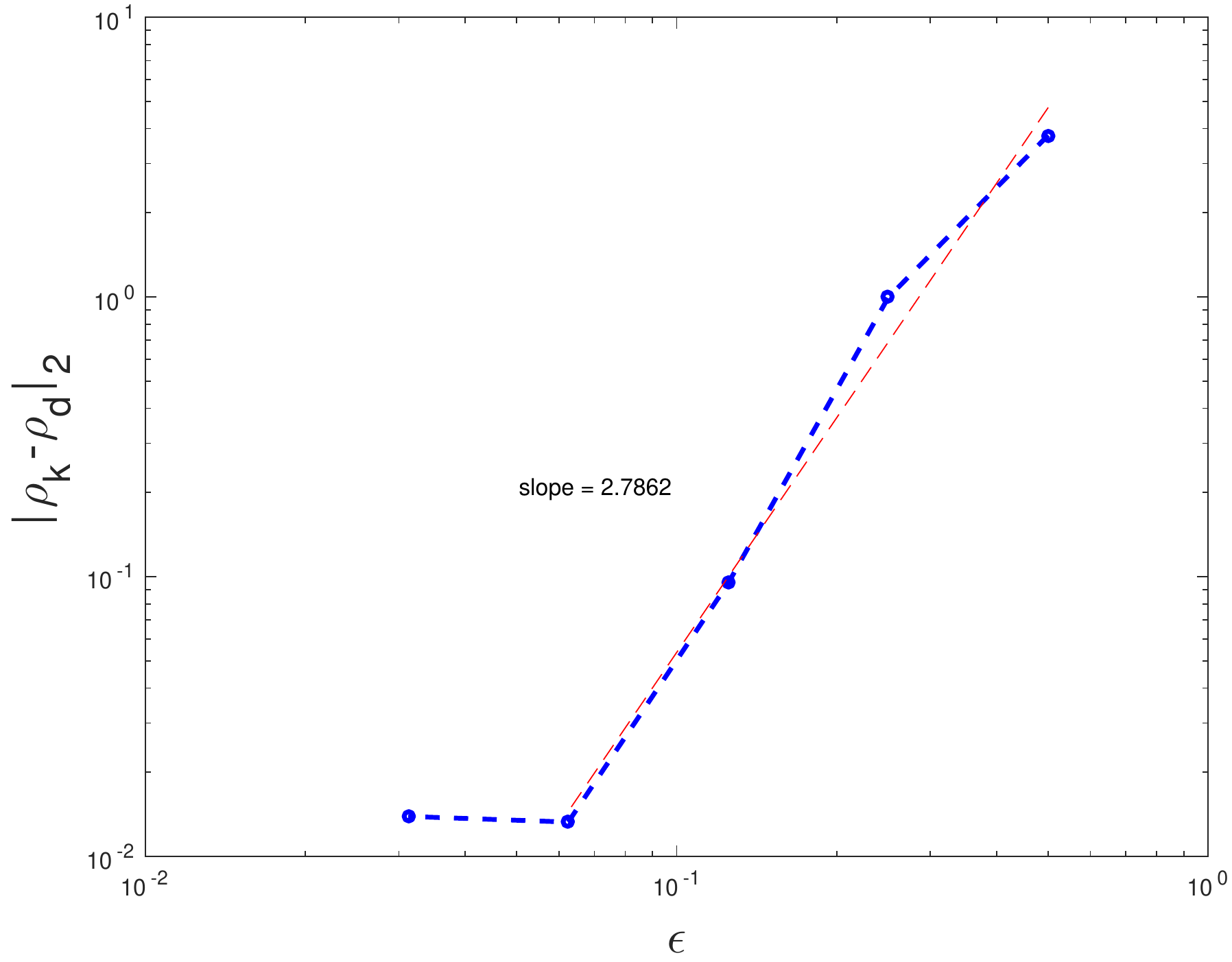}
	\caption{Example V. Left:  $l^2$ distance between $f$ and $\rho$ (\ref{AP error}) versus time using our kinetic solver. Right: $l^2$ distance between $\rho_k$, the density of the kinetic and $\rho_d$, the density to the diffusion equation. The error saturates at the last point of computation. }\label{fig: AP error}
\end{figure}

\item Example VI: Initial data is the same as in (\ref{ic-2D}) and scattering cross-section takes the following form
\begin{eqnarray} \label{exp3-xsection}
\sigma(x,y) = \left\{ \begin{array}{cc} 0.02 \quad &(x,y) \in [0.25,0.35] \times[0.25,0.35] \cup [0.65,0.75] \times[0.65,0.75] 
\\ 1 \quad & \textrm{elsewhere}
\end{array} \right.
\end{eqnarray}
which again contains both strong and weak scattering regimes (see Fig.~\ref{fig: 2D1}) on the upper left.  Let $\eps = 1e-4$, we compare the solution using our scheme with the solution to the diffusion limit at time $t_\textrm{max} = 0.1$, which is displayed in Fig.~\ref{fig: 2D1} with good agreement.

\begin{figure}[!h] 
      \includegraphics[width = 0.45\textwidth]{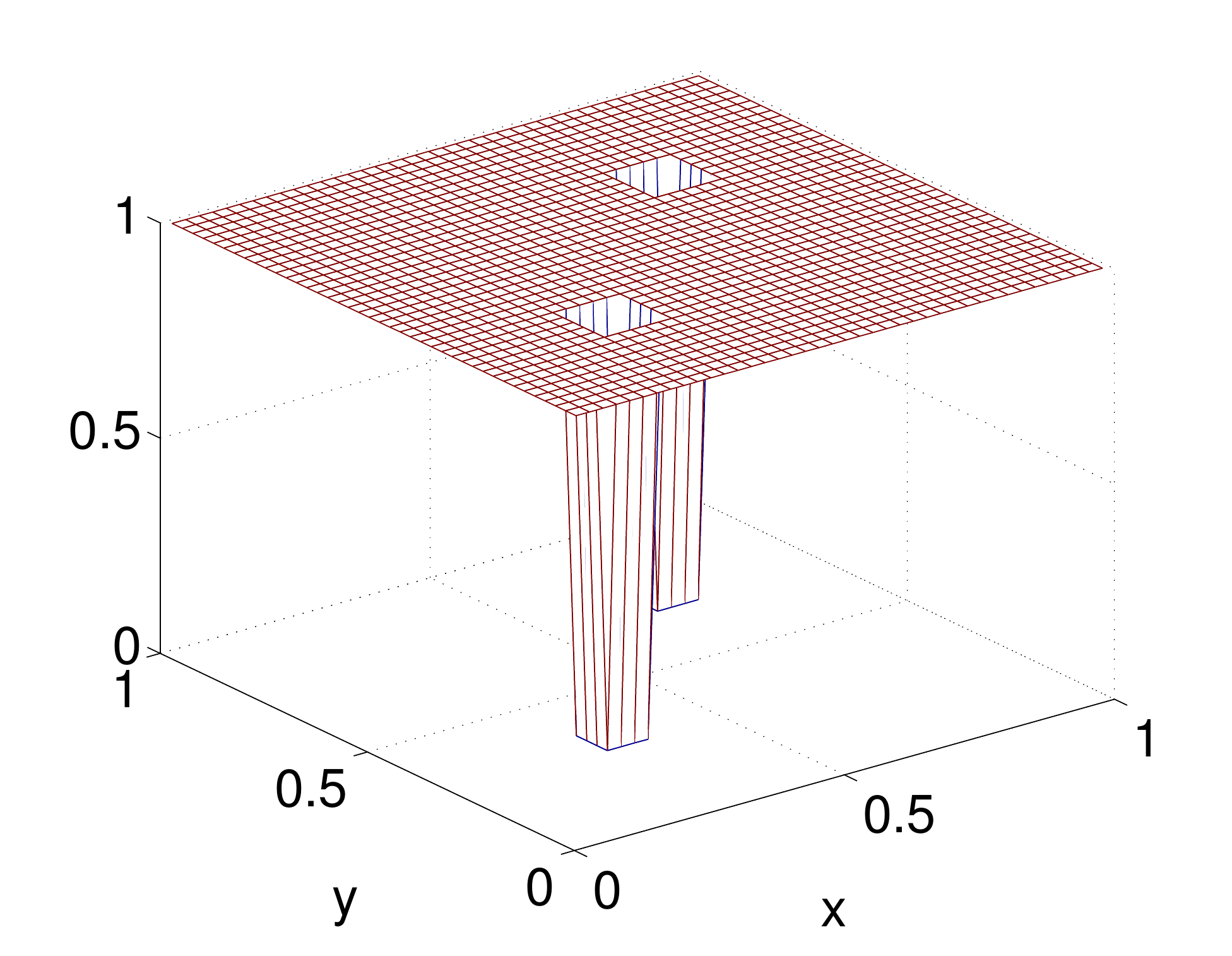}
       \includegraphics[width=0.45\textwidth]{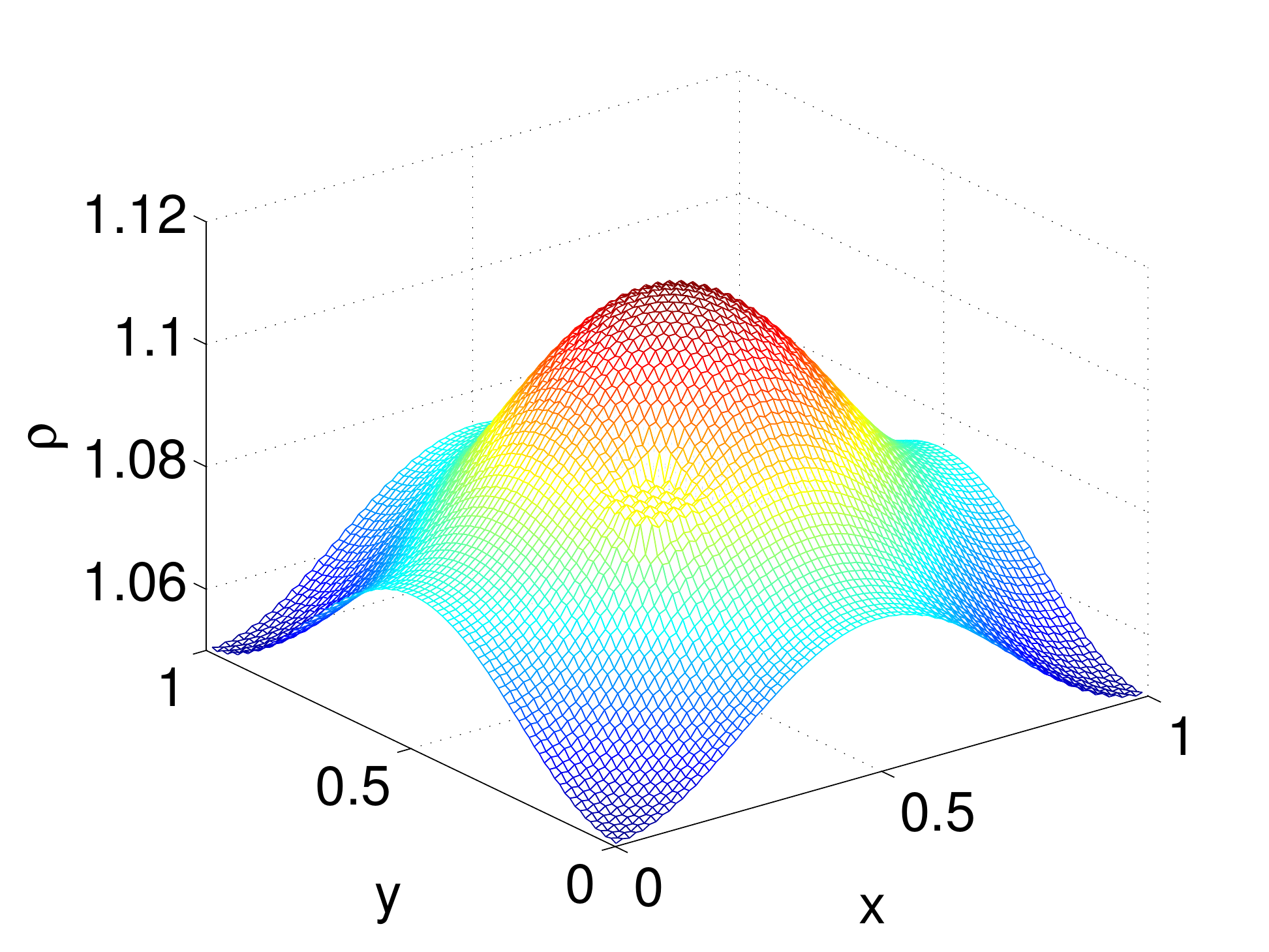}
       \\
	\includegraphics[width=0.45\textwidth]{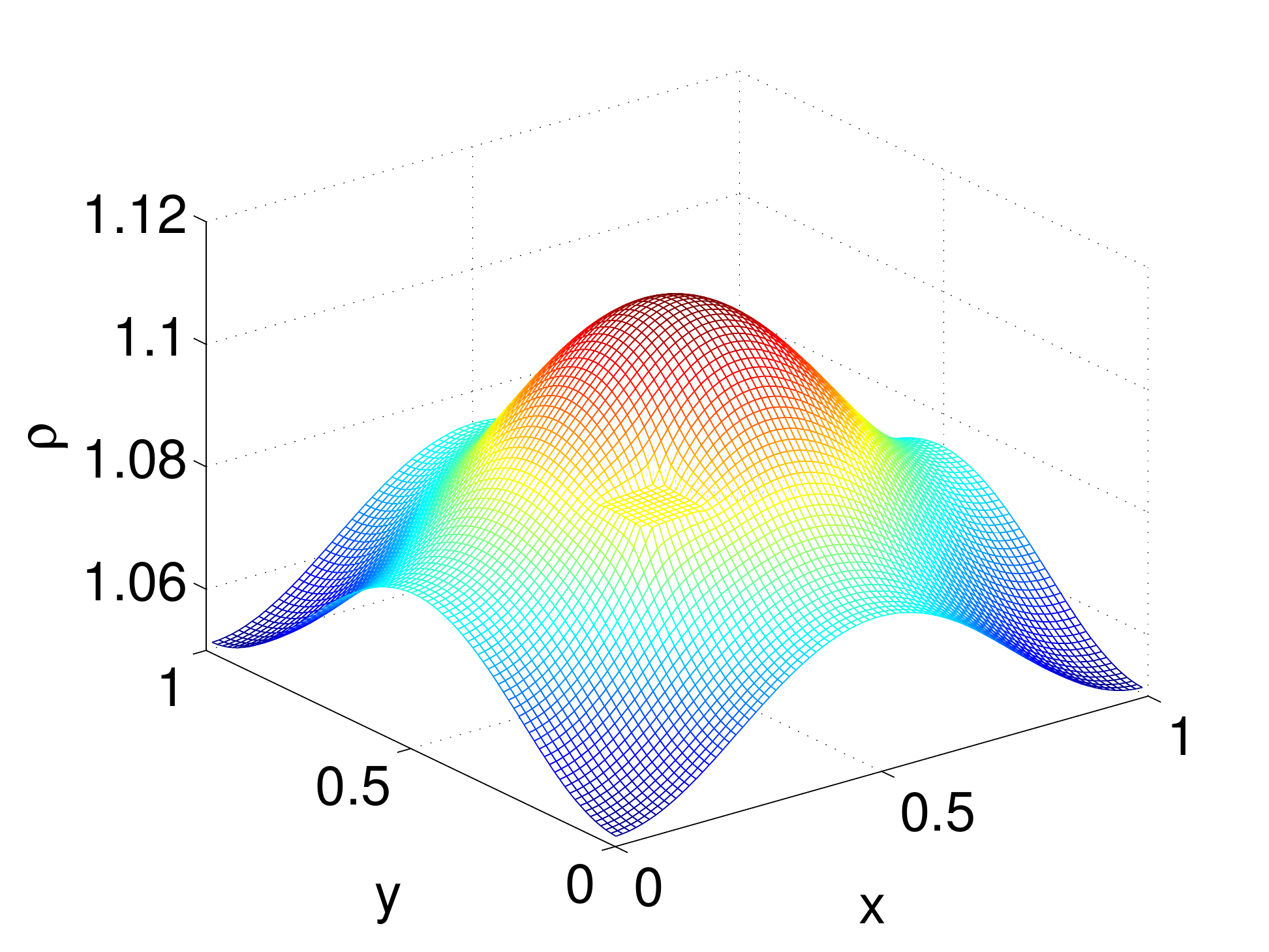}	
	\includegraphics[width=0.45\textwidth]{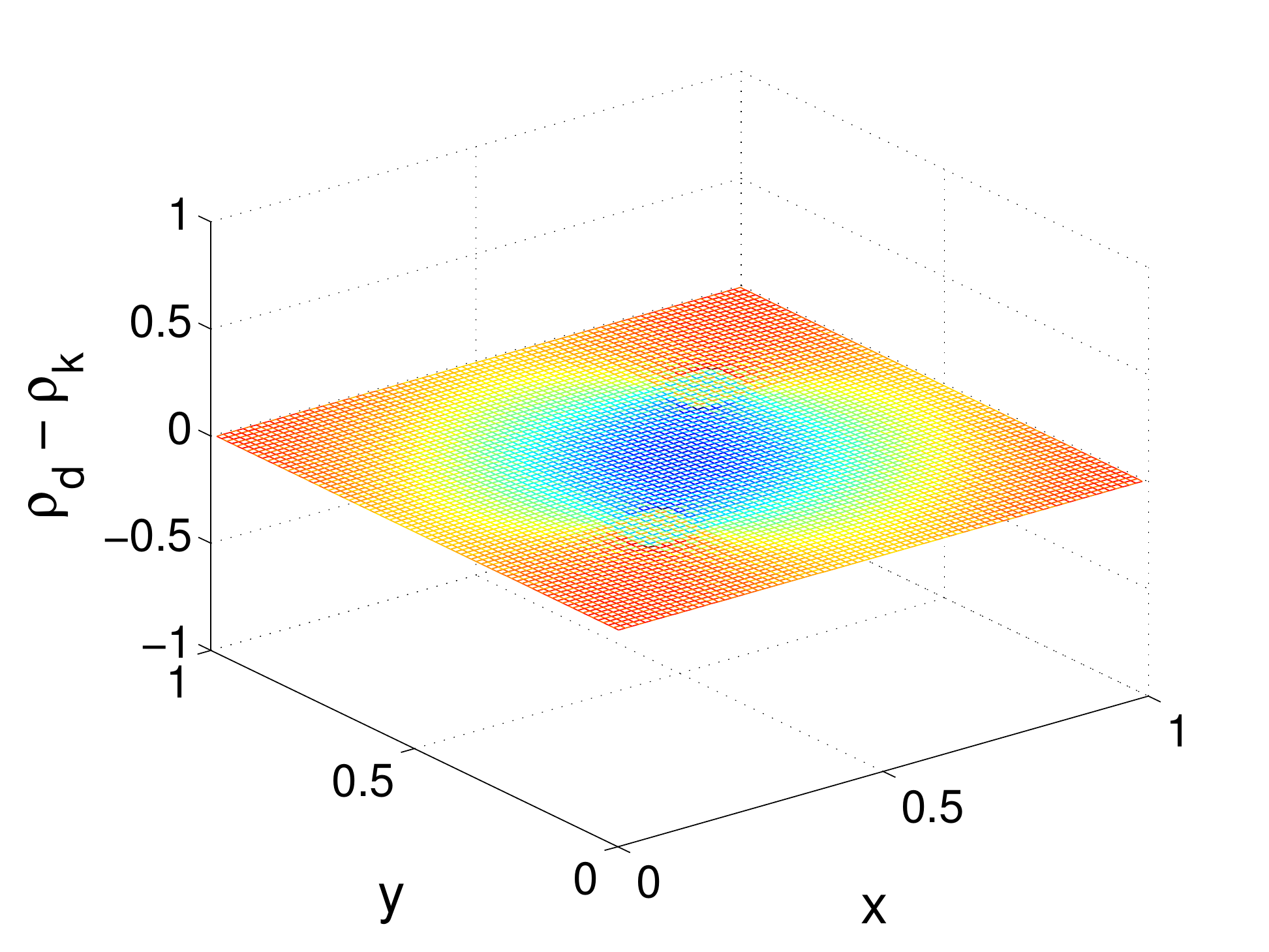}
	\caption{Example VI. Upper Left: Cross section (\ref{exp3-xsection}). Upper right: solution with our kinetic solver. Lower left : solution to the diffusion limit (\ref{diff_2D}). Lower right: difference between the two. Here $N_x = N_y = 80$, $N_v = 10$ and $\eps = 10^{-4}$. }\label{fig: 2D1}
\end{figure}

\end{itemize}

\section{Conclusion}
In this paper, we designed a fast solver for the fully implicit treatment of the linear transport equation. When the scattering effect is strong, this equation exhibits diffusive scaling such that both the convection and collision become stiff. On the other hand, when the scattering is very weak, the photon dynamics will be dominated by a free transport at the speed of light. In either case, numerically solving the equation requires a special care to deal with the stiffness. The fully implicit time discretization we considered here effectively treat the stiffness without resolving the mesh size, but at the cost of generating a large algebraic system that needs to invert, which is also ill-conditioned and not necessarily symmetric. We propose an efficient pre-conditioner which significantly improve and condition number and allows matrix-free treatment. The key ingredient is to use the spectral structure for the collision operator, which is also the source of ill-conditioning, to compute the pre-conditioner. We also reformulate the system via an even-odd parity so that the resulting linear system is symmetric and positive definite that can be inverted using conjugate gradient method with ease. A asymmetric version is also available and can be inverted through Krylov method such as GMRES. A major benefit of our new method is that it does not depend on the specific form of spatial or angular discretization, therefore it can be used with great generality. In the near future, we will generalize this method to nonlinear transport equation. 

\bibliographystyle{siam}
\bibliography{RTE}

\begin{bibdiv}
\begin{biblist}

\bib{ABDH95}{article}{
      author={Ashby, S.},
      author={Brown, P.},
      author={Dorr, M.},
      author={Hindmarsh, A.},
       title={A linear algebraic analysis of diffusion synthetic acceleration
  for the {B}oltzmann transport equation},
        date={1995},
     journal={SIAM J. Numer. Anal.},
      volume={32},
       pages={128\ndash 178},
}

\bib{AL02}{article}{
      author={Adams, M.},
      author={Larsen, E.},
       title={Fast iterative methods for discrete-ordinate particle transport
  calculations},
        date={2002},
     journal={Prog. Nucl. Energ.},
      volume={40},
      number={1},
       pages={3\ndash 159},
}

\bib{Azmy02}{article}{
      author={Azmy, Y.},
       title={Unconditionally stable and robust adjacent-cell diffusive
  preconditioning of weighted-difference particle transport methods is
  impossible},
        date={2002},
     journal={J. Comput. Phys.},
      volume={182},
       pages={213},
}

\bib{BL01}{article}{
      author={Brantley, P.},
      author={Larsen, E.},
       title={The simplified {P3} approximation},
        date={2001},
     journal={Nucl. Sci. Eng.},
      volume={134},
       pages={1\ndash 21},
}

\bib{BPR13}{article}{
      author={Boscarino, S.},
      author={Pareschi, L.},
      author={Russo, G.},
       title={Implicit-explicit {R}unge-{K}utta scheme for hyperbolic systems
  and kinetic equations in the diffusion limit},
        date={2013},
     journal={SIAM J. Sci. Comput.},
      volume={35},
       pages={22\ndash 51},
}

\bib{Brown95}{article}{
      author={Brown, P.},
       title={A linear algebraic development of diffusion synthetic
  acceleration for three-dimensional transport equations},
        date={1995},
     journal={SIAM J. Numer. Anal.},
      volume={32},
       pages={179\ndash 214},
}

\bib{FC71}{article}{
      author={Fleck, J.A.},
      author={Cummings, J.D.},
       title={An implicit {M}onte {C}arlo scheme for calculating time and
  frequency dependent nonlinear radiation transpor},
        date={1971},
     journal={J. Comput. Phys.},
      volume={8},
       pages={313\ndash 342},
}

\bib{FKLY07}{article}{
      author={Frank, M.},
      author={Klar, A.},
      author={Larsen, E.},
      author={Yasuda, S.},
       title={Time-dependent simplified pn approximation to the equations of
  radiative transfer},
        date={2007},
     journal={J. Comput. Phys.},
      volume={226},
       pages={2289\ndash 2305},
}

\bib{FS11}{article}{
      author={Frank, M.},
      author={Seibold, B.},
       title={Optimal prediction for radiative transfer: a new perspective on
  moment closure},
        date={2011},
     journal={Kinet. Relat. Models},
      volume={31},
       pages={717\ndash 733},
}

\bib{GHP99}{article}{
      author={Guthrie, B.},
      author={Holloway, J.},
      author={Patton, B.},
       title={G{M}{R}{E}{S} as a multi-step transport sweep accelerator},
        date={1999},
     journal={Trans. Theory and Stat. Phys.},
      volume={28},
      number={1},
       pages={83\ndash 102},
}

\bib{GJL99}{article}{
      author={Golse, F.},
      author={Jin, S.},
      author={Levermore, D.},
       title={The convergence of numerical transfer schemes in diffusive
  regimes i: The discrete-ordinate method},
        date={1999},
     journal={SIAM J. Numer. Anal.},
      volume={36},
       pages={1333\ndash 1369},
}

\bib{Hauck11}{article}{
      author={Hauck, C.},
       title={High-order entropy-based closures for linear transport in slab
  geometries},
        date={2011},
     journal={Commun. Math. Sci.},
      volume={9},
       pages={187\ndash 205},
}

\bib{JL93}{article}{
      author={Jin, S.},
      author={Levermore, D.},
       title={Fully discrete numerical transfer in diffusive regimes},
        date={1993},
     journal={Transp. Theory Stat. Phys.},
      volume={22},
       pages={739\ndash 791},
}

\bib{JPT00}{article}{
      author={Jin, S.},
      author={Pareschi, L.},
      author={Toscani, G.},
       title={Uniformly accurate diffusive relaxation schemes for multiscale
  transport equations},
        date={2000},
     journal={SIAM J. Num. Anal.},
      volume={38},
       pages={913\ndash 936},
}

\bib{KFJ15}{article}{
      author={Kupper, K.},
      author={Frank, M.},
      author={Jin, S.},
       title={An asymptotic preserving 2--d staggered grid method for
  multiscale transport equations},
        date={2015},
     journal={submitted},
}

\bib{Klar98}{article}{
      author={Klar, A.},
       title={An asymptotic-induced scheme for nonstationary transport
  equations in the diffusive limit},
        date={1998},
     journal={SIAM J. Numer. Anal.},
      volume={35},
      number={6},
       pages={1097\ndash 1094},
}

\bib{Larsen80}{article}{
      author={Larsen, E.},
       title={Diffusion theory as an asymptotic limit of transport theory for
  nearly critical systems with small mean free paths},
        date={1980},
     journal={Ann. Nucl. Energy},
      volume={7},
       pages={249\ndash 255},
}

\bib{LM08}{article}{
      author={Lemou, M.},
      author={Mieussens, L.},
       title={New asymptotic preserving scheme based on micro-macro formulation
  for linear kinetic equations in the diffusion limit},
        date={2008},
     journal={SIAM J. Sci. Comput.},
      volume={31},
       pages={334\ndash 368},
}

\bib{LM89}{article}{
      author={Larsen, E.},
      author={Morel, J.},
       title={Asymptotic solutions of numerical transport problems in optically
  thick, diffusive regimes. ii},
        date={1989},
     journal={J. Comput. Phys.},
      volume={83},
       pages={212\ndash 236},
}

\bib{Mieussens13}{article}{
      author={Mieussens, L.},
       title={On the asymptotic preserving property o fate unified gas kinetic
  scheme for the diffusion limit of linear kinetic modle},
        date={2013},
     journal={J. Comput. Phys.},
      volume={253},
       pages={138\ndash 156},
}

\bib{MWLP03}{article}{
      author={Morel, J.},
      author={Wareing, T.},
      author={Lowrie, R.},
      author={Parsons, D.},
       title={Analysis of ray-effect mitigation techniques},
        date={2003},
     journal={Nucl. Sci. Eng.},
      volume={144},
      number={1},
       pages={1\ndash 22},
}

\bib{OHF12}{article}{
      author={Olbrant, E.},
      author={Hauck, C.},
      author={Frank, M.},
       title={A realizability-preserving discontinuous {G}alerkin method for
  the m1 model of radiative transfer},
        date={2012},
     journal={J. Comput. Phys.},
      volume={231},
       pages={5612\ndash 5639},
}

\bib{OLS13}{article}{
      author={Olbrant, E.},
      author={Larsen, E.},
      author={Frank, M.},
      author={Seibold, B.},
       title={Asymptotic derivation and numerical investigation of
  time-dependent simplified {$P_N$} equations},
        date={2012},
     journal={J. Comput. Phys.},
      volume={238},
      number={1},
       pages={315\ndash 336},
}

\bib{Pomraning93}{article}{
      author={Pomraning, G.},
       title={Asymptotic and variational derivations of the simplified {PN}
  equations},
        date={1993},
     journal={Ann. Nucl. Energy},
      volume={20},
       pages={623\ndash 637},
}

\bib{SJX15}{article}{
      author={Sun, W.},
      author={Jiang, S.},
      author={Xu, K.},
       title={An asymptotic preserving unified gas kinetic scheme for gray
  radiative transfer equations},
        date={2015},
     journal={J. Comput. Phys.},
      volume={285},
       pages={265\ndash 279},
}

\bib{TL96}{article}{
      author={Tomasevic, D.},
      author={Larsen, E.},
       title={The simplified {P2} approximation},
        date={1996},
     journal={Nucl. Sci. Eng.},
      volume={122},
       pages={309\ndash 325},
}

\bib{WMMP01}{article}{
      author={Wareing, T.~A.},
      author={McGhee, J.},
      author={Morel, J.},
      author={Pautz, S.},
       title={Discontinuous finite element $s_n$ methods on three-dimensional
  unstructured grids},
        date={2001},
     journal={Nucl. Sci. Eng.},
      volume={138},
      number={3},
       pages={256\ndash 268},
}

\bib{WWM04}{article}{
      author={Warsa, J.},
      author={Wareing, T.},
      author={Morel, J.},
       title={Krylov iterative methods and the degraded effectiveness of
  diffusion synthetic acceleration for multidimensional {S}n calculations in
  problems with material discontinuities},
        date={2004},
     journal={Nuclear Math. Comput. Sci.},
      volume={147},
       pages={218\ndash 248},
}

\end{biblist}
\end{bibdiv}

\end{document}